\documentclass[11pt]{amsart}
\usepackage{amsmath,amssymb,amsthm,amsfonts}
\usepackage{paralist}
\usepackage{booktabs}
\usepackage[usenames,dvipsnames]{color}
\usepackage{comment}
\usepackage{graphicx,epsf}
\usepackage{graphics}
\usepackage{epsfig}
\usepackage{epstopdf}
\usepackage{psfrag}
\usepackage{latexsym}
\usepackage[latin1]{inputenc}
\usepackage[shortlabels]{enumitem}
\usepackage{thmtools}
\usepackage{url}
\usepackage[all,cmtip]{xy}
\usepackage{hyperref} %hyperref wants to be loaded last
\usepackage{pifont}
\usepackage{mathdots}

\newcommand{\Dyck}{\mathsf{Dyck}}

\newcommand{\ASM}{{\rm {\mathrm{{ASM}}}}}
\newcommand{\SLT}{{\rm {\mathrm{{FSLT}}}}}
\newcommand{\TSPP}{{\rm {\mathrm{{TSPP}}}}}

\title[Complexity problems in enumerative combinatorics]
{Complexity problems in enumerative combinatorics}

\author[Igor Pak]{
Igor Pak$^\star$}

\thanks{\today}
\thanks{\thinspace ${\hspace{-.45ex}}^\star$Department of Mathematics,
UCLA, Los Angeles, CA~90095.
\hskip.06cm
Email:
\hskip.06cm
\texttt{pak@math.ucla.edu}}

\newcommand{\dinv}{\mathrm{dinv}}

\newcommand{\sgn}{\mathrm{sgn}}

\newcommand{\bone}{\mathbf{1}}

\newcommand{\SYT}{\operatorname{SYT}}

\DeclareMathOperator{\area}{area}

\newtheorem{thm}{Theorem}[section]

\newtheorem{claim}[thm]{Claim}
\newtheorem{cor}[thm]{Corollary}
\newtheorem{prop}[thm]{Proposition}
\newtheorem{conj}[thm]{Conjecture}

\newtheorem{question}[thm]{Question}

\newtheorem{rem}[thm]{Remark}

\newtheorem{op}[thm]{Open Problem}
\newtheorem{ex}[thm]{Example}

\numberwithin{equation}{section} % requires package amsthm

\def\zz{\mathbb Z}
\def\nn{\mathbb N}
\def\gg{\mathbb G}

\def\rr{\mathbb R}
\def\qqq{\mathbb Q}

\def\ov{\overline}

\def\Ga{\Gamma}

\def\la{\lambda}
\def\ga{\gamma}
\def\si{\sigma}

\def\ep{\epsilon}
\def\al{\alpha}
\def\be{\beta}
\def\om{\omega}

\def\vp{\varphi}
\def\CR{\mathcal R}

\def\cB{\mathcal B}
\def\ca{\mathcal A}
\def\cA{\mathcal A}
\def\cb{\mathcal B}

\def\cf{\mathcal F}
\def\cF{\mathcal F}

\def\cP{\mathcal P}

\def\cQ{\mathcal Q}
\def\cR{\mathcal R}

\def\ssu{\subset}

\def\<{\langle}
\def\>{\rangle}

\def\GL{ {\text {\rm GL} } }

\def\BS{ {\text {\rm BS} } }

\def\PSL{ {\text {\rm PSL} } }

\def\bT{{\text {\bf T} } }

\def\rP{\text{{\rm \textrm{P}}}}
\def\rQ{\text{{\rm \textrm{Q}}}}
\def\rCr{\text{{\rm \textrm{Cr}}}}

\def\Av{\text{{\rm \textsf{A}}}}

\def\rCT{ {\text {\rm CT}} }

\def\0{{\mathbf 0}}

\def\SP{\textup{\textsf{\#P}}}
\def\SEXP{\textup{\textsf{\#EXP}}}
\def\EXP{\textup{\textsf{EXP}}}
\def\NEXP{\textup{\textsf{NEXP}}}
\def\SP{\textup{\textsf{\#P}}}
\def\NP{\textup{\textsf{NP}}}

\def\GapP{\textup{\textsf{GapP}}}
\def\coNP{\textup{\textsf{co-NP}}}
\def\FP{\textup{\textsf{FP}}}
\def\PP{\textup{\textsf{P}}}
\def\PPA{\textup{\textsf{PPA}}}

\def\poly{\textup{\textrm{poly}}}
\def\PP{{\textup{\textsf{P}}}}
\def\FP{{\textup{\textsf{FP}}}}
\def\SP{{\textup{\textsf{\#P}}}}
\def\NP{{\textup{\textsf{NP}}}}

\def\bcP{{\pi}}
\def\bcQ{{\om}}

\def\.{\hskip.06cm}
\def\ts{\hskip.03cm}

\def\nin{\noindent}

\def\WI{\textup{\small $($\textrm{W1}$)$}}
\def\WII{\textup{\small $($\textrm{W2}$)$}}
\def\WIII{\textup{\small $($\textrm{W3}$)$}}
\def\WIV{\textup{\small $($\textrm{W4}$)$}}

\def\CI{\textup{\small $($\textrm{C1}$)$}}
\def\CII{\textup{\small $($\textrm{C2}$)$}}

\begin{document}

\begin{abstract}
We give a broad survey of recent results in Enumerative
Combinatorics and their complexity aspects.
\end{abstract}

%\ytableausetup{smalltableaux}

\maketitle
%\tableofcontents

%----------------------------------------------------------------
\section*{Introduction} \label{sec:intro}
%----------------------------------------------------------------

\nin
The subject of enumerative combinatorics is both classical and modern.
It is classical, as the basic counting questions go back millennia; yet
it is modern in the use of a large variety of the latest ideas and
technical tools from across many areas of mathematics.  The remarkable
successes from the last few decades have been widely publicized; yet they
come at a price, as one wonders if there is anything left to explore.
In fact, are there enumerative problems that cannot be resolved with
existing technology?  In this paper we present many challenges in the
field from the computational complexity point of view, and describe
how recent results fit into the story.

\smallskip

Let us first divide the problems into three major classes.  This division
is not as neat as it may seem, as there are problems which fit into multiple
or none of the classes, especially if they come from other areas. Still,
it provides us with a good starting point.

\medskip

\nin
$(1)$ \textbf{Formulas.} Let $\cP$ be a set of combinatorial
objects -- think of trees, words, permutations, Young tableaux, etc.
Such objects often come with a parameter~$n$ corresponding
to the size of the objects.  Let $\cP_n$ be the set of objects
of size~$n$.  Find a formula for $|\cP_n|$.

\smallskip

\nin
$(2)$ \textbf{Bijections.} Now let $\cP$ and $\cQ$ be two sets of
(possibly very different) combinatorial objects.  Suppose that
you know (or at least suspect) that \ts $|\cP_n|=|\cQ_n|$.
Find an explicit bijection $\vp: \cP_n \to \cQ_n$.

\smallskip

\nin
$(3)$ \textbf{Combinatorial interpretations.} Now suppose there
is an integer sequence $\{a_n\}$ given by a formula. Suppose that
you know (or at least suspect) that \ts $a_n\ge 0$ \ts for all~$n$.
Find a combinatorial interpretation of $a_n$, i.e.\ a set of
combinatorial objects $\cP$ such that $|\cP_n|=a_n$.

\medskip

People in the area are well skilled in both resolving and justifying
these problems.  Indeed, a formula is a good thing to have in case
one needs to compute $|\cP_n|$ explicitly for large~$n$,
find the asymptotics, gauge the structural complexity of the
objects, etc. A bijection between a complicated set $\cP$ and
a simpler set $\cQ$ is an even better thing to have,
as it allows one to better understand the nature of~$\cP$,
do a refined counting of $\cP_n$ with respect to various
statistics, generate elements of $\cP_n$ at random, etc.
Finally, a combinatorial interpretation is an excellent
first step that allows one to proceed to $(1)$ and
then~$(2)$, or at least obtain some useful estimates for~$a_n$.

\smallskip

Here is a more difficult part, which comes in the form of
inquisitive questions in each case:

\smallskip

\nin
$(1')$ \ts What is a formula?  What happens if there is no formula?
Can you prove there isn't one?  How do you even formalize the last
question if you don't know the answer to the first?

\smallskip

\nin
$(2')$ \ts There are, obviously, $|\cP_n|!$ bijections $\vp: \cP_n\to \cQ_n$,
so you must want a particular one, or at least one with certain properties?
Is there a ``canonical'' bijection, or at least the one you want the most?
What if there isn't a good bijection by whatever measure -- can you
prove that?  Can you even formalize that?

\smallskip

\nin
$(3')$ \ts Again, what do you do in the case when there isn't a combinatorial
interpretation?  Can you formally prove a negative result so that others
stop pursuing these problems?

\smallskip

We give some formal answers to these questions, at least in several interesting
special cases.  As the reader will see, the complexity approach brings
some clarity in most cases.  But to give the answers we first need to
explain the nature of combinatorial objects in each case, and to review
the literature.  That is the goal of this survey.

\medskip

Before we conclude the introduction, let us quote Gian-Carlo Rota, one of the
founding fathers of modern enumerative combinatorics:

\medskip

\begin{center}\begin{minipage}{12cm}%
\emph{``Combinatorics is an honest subject. No ad\`{e}les, no sigma-algebras.
You count balls in a box, and you either have the right number or you haven't.
You get the feeling that the result you have discovered is forever, because
it's concrete. Other branches of mathematics are not so clear-cut,
$[$\ts..\ts$]$ never fully convincing: you don't get a feeling of
having done an honest day's work. Don't get the wrong idea---combinatorics
is not just putting balls into boxes. Counting finite sets can be a
highbrow undertaking, with sophisticated techniques.''~\cite{Rota-Los-Alamos}
}\end{minipage}\end{center}

\medskip

\nin
Rota is right, of course, historically speaking.  When the result is positive,
it's ``forever'' indeed, and this partly explains the glamour of $(1)-(3)$.
But when the result is negative, when questions $(1')-(3')$ are addressed,
this certainty disappears.  Our current understanding of a ``formula''
and a ``good bijection'' can change in the future, perhaps fundamentally,
as it has changed in the past.  Forever these results are certainly not.
In fact, when the complexity assumptions such as \ts
$\PP\ne \NP$, \ts $\FP\ne \SP$, etc.\ become essential, one must learn to
accept the uncertainty and learn to navigate in this new environment\ldots \. at
least until computational complexity brings more clarity to these matters.

\medskip

\section{What is a formula?}

\subsection{Basic examples}  % It is worth writing out a list of examples.
We start with the Fibonacci numbers \cite[\href{https://oeis.org/A000045}{A000045}]{OEIS}:
\begin{equation}\label{eq:Fib1}
  F_{n} \, = \, F_{n-1} + F_{n-2}, \ \quad F_0=F_1=1
\end{equation}
\begin{equation}\label{eq:Fib2}
\qquad F_n \, = \, \sum_{i=0}^{\lfloor n/2\rfloor}  \, \binom{n-i}{i}
\end{equation}
\begin{equation}\label{eq:Fib3}
 F_n \, = \, \frac{1}{\sqrt{5}}\.\Bigr(\ts \phi^n \. - \. (-\phi)^{-n}\ts\Bigr)\ts, \quad \text{where}
\quad \phi\. = \. \frac{1+\sqrt{5}}2
\end{equation}
\begin{equation}\label{eq:Fib4}
 F_n \, = \, \bigl(A^n\bigr)_{2,2} \ , \quad \text{where}
\quad
A\. = \. \begin{pmatrix} 0 & 1 \\
1 & 1
\end{pmatrix}.
\end{equation}
Equation~\eqref{eq:Fib1} is usually presented as a definition, but can also be used
to compute $F_n$ in \poly$(n)$ time.  Equation~\eqref{eq:Fib2} is useful to place
Fibonacci numbers in the hierarchy of integer sequences (see below).
Equation~\eqref{eq:Fib3} is useful to obtain asymptotics, and equation~\eqref{eq:Fib4} gives
a fast algorithm for computing~$F_n$ (by repeated squaring). \emph{The moral}:
there is no one notion of a ``good formula'', as different equations have
different uses.

Let us consider a few more plausible formula candidates:
\begin{equation} D_n \, = \, \left[\left[n!/e\right]\right], \quad
\text{where} \ \, [[x]] \, \ \text{denotes the nearest integer}
\label{eq:formula-D}
\end{equation}
\begin{equation}
C_n \, = \, [t^n] \ts \frac{1-\sqrt{1-4t}}{2t}
\label{eq:formula-C}
\end{equation}
\begin{equation}
E_n \, = \, n! \ts\cdot\ts [t^n] \. y(t), \quad \text{where} \ \ 2\ts y' \. = \. 1\ts +\ts y^2, \ \, y(0)=1
\label{eq:formula-E}
\end{equation}
\begin{equation}
T_n \, = \, (n-1)! \ts \cdot \ts [t^n] \. z(t), \quad \text{where} \ \ z \. = \. t\ts e^{t\ts e^{t\ts e^{t\ts e^{\iddots}}}}.
\label{eq:formula-T}
\end{equation}
Here $D_n$ is the number of \emph{derangements}
(fixed-point-free permutations in $S_n$), \ts $C_n$ is the
\emph{Catalan number} (the number of binary trees with $n$ vertices),
$E_n$ is the \emph{Euler number} (the number of \emph{alternating permutations}
$\si(1)<\si(2)>\si(3)<\si(4)>\ldots $ in $S_n$), and $T_n$ is the
\emph{Cayley number} (the number of spanning trees in the complete graph $K_{n}$).
Here and everywhere below we use $[t^n] \ts F(t)$ \ts to denote the
coefficient of $t^n$ in~$F(t)$.

Note that in each case above, there are better formulas for applications:
\begin{equation} \label{eq:formula-Derangements}
  D_n \, = \, n!\. \sum_{k=0}^n \.\frac{(-1)^k}{k!}  \\
\end{equation}
\begin{equation}\label{eq:formula-Catalan}
  C_n \, = \, \frac{1}{n+1}\binom{2n}{n}\\
\end{equation}
\begin{equation} \label{eq:formula-Euler}
E_n \, = \, n! \cdot [t^n] \. y(t), \quad \text{where} \ \ y(t)\. = \. \tan(t) + \sec(t) \\
\end{equation}
\begin{equation}\label{eq:formula-Cayley}
 T_n \, = \, n^{n-2}\ts.
\end{equation}
In all four cases, the corresponding formulas are equivalent by mathematical
reasoning. Whether or not you accept \eqref{eq:formula-D}--\eqref{eq:formula-T}~as formulas,
it is their meaning that's important, not their form.

\smallskip

Finally, consider the following equations for the \emph{number of partitions}~$p(n)$, and
\emph{prime-counting function}~$\pi(n)$:
\begin{equation}\label{eq:formula-part}  p(n) \, = \, [t^n] \. \prod_{i=1}^\infty \frac{1}{1-t^i}
\end{equation}
\begin{equation}\label{eq:formula-primes}
% \aligned
\pi(n) \, = \, \sum_{k=2}^{n} \. \left(\ts\left\lfloor \frac{(k-1)!+1}{k}\right\rfloor
\. - \. \left\lfloor \frac{(k-1)!}{k}\right\rfloor \ts\right).
%\endaligned
\end{equation}
Equation~\eqref{eq:formula-part} is due to Euler~(1748), and had profound
implications in number theory and combinatorics, initiating the whole
area of \emph{partition theory} (see e.g.~\cite{And}).
Equation~\eqref{eq:formula-primes} follows easily from Wilson's theorem
(see e.g.~\cite[$\S$1.2.5]{CPom}).  Esthetic value aside,
both equations are largely unhelpful for computing purposes and follow directly
from definitions. Indeed, the former is equivalent to the standard counting
algorithm (\emph{dynamic programming}), while the latter is an iterated
divisibility testing in disguise.

In summary, we see that the notion of ``good formula'' is neither syntactic nor semantic.
One needs to make a choice depending on the application.

\smallskip

\subsection{Wilfian formulas}  In his pioneering 1982 paper~\cite{Wilf},
Wilf proposed to judge a formula from the complexity point of view.
He suggested two definitions of ``good formulas'' for computing an
integer sequence $\{a_n\}$:

\smallskip

% \nin
\quad \WI \ There is an algorithm that computes $a_n$ in time \ts \poly$(n)$.

\smallskip

% \nin
\quad \WII \ There is an algorithm that computes $a_n$ in time \ts $o(a_n)$.

\smallskip

\nin
In the literature, such algorithms are called sometimes \emph{Wilfian formulas}.
Note that~\WI~is aimed (but not restricted) to apply to sequences $\{a_n\}$
of at most exponential growth \ts $a_n=\exp O(n^c)$, while~\WII~for $\{a_n\}$
of at most polynomial growth. See e.g.~\cite{GP-words,FS} for more
on growth of sequences.

Going over our list of examples we conclude that~\eqref{eq:Fib1}, \eqref{eq:Fib2},
\eqref{eq:Fib4}, \eqref{eq:formula-Derangements}, \eqref{eq:formula-Catalan} and~\eqref{eq:formula-Cayley}
are all transparently Wilfian of type~\WI.  Equations~\eqref{eq:Fib3}, \eqref{eq:formula-C},
\eqref{eq:formula-E} and~\eqref{eq:formula-Euler} are Wilfian of type~\WI~in a less obvious but
routine way (see below).  Equations~\eqref{eq:Fib3} and~\eqref{eq:formula-D} do give rise
to ad hoc \ts \poly$(n)$ \ts algorithms, but care must be applied when dealing with irrational
numbers. E.g., one must avoid circularity, such as when computing $n$-th prime $p_n$ by using the
\emph{prime constant} \ts $\sum_{n} 1/2^{p_n}$, see e.g.~\cite[$\S$1.2.5]{CP} and
\cite[\href{https://oeis.org/A051006}{A051006}]{OEIS}.
Finally, equation~\eqref{eq:formula-T} is not Wilfian of type~\WI, while~\eqref{eq:formula-primes}
is not Wilfian of type~\WII.

%\nin
Let us add two more notions of a ``good formula'' in the same spirit, both of which
are somewhat analogous but more useful than~\WII:

\smallskip

% \nin
\quad \WIII \ There is an algorithm that computes $a_n$ in time \ts \poly$(\log n)$.

\smallskip

\quad \WIV \ There is an algorithm that computes $a_n$ in time \ts $n^{o(1)}$.

\smallskip

\nin
Now, for a \emph{combinatorial sequence} $\{a_n\}$ one can ask if there is a
Wilfian formula.  In the original paper~\cite{Wilf} an explicit example is given:

\begin{conj}[Wilf] \label{c:Wilf}
Let $a_n$ be the number of unlabeled graphs on $n$ vertices.  Then $\{a_n\}$ has
no Wilfian formula of type~\WI.
\end{conj}

See~\cite[\href{https://oeis.org/A000088}{A000088}]{OEIS} for this sequence.
Note that by the classical Erd\H{o}s--R\'enyi result~\cite{ER}
(see also~\cite[$\S$1.6]{Bab-handbook}), we have \ts
$a_n \sim 2^{\binom{n}{2}}/n!$, so the problem is not
approximating~$a_n$, but computing it exactly.
For comparison, the sequence $\{c_n\}$ of the number of connected
(labeled) graphs does have a Wilfian formula of type~\WI:
$$
c_n \, = \, 2^{\binom{n}{2}} \. - \. \frac{1}{n} \.
\sum_{k=1}^{n-1} \. k\ts\binom{n}{k} \ts 2^{\binom{n-k}{2}} \ts c_k
$$
(see~\cite[\href{https://oeis.org/A001187}{A001187}]{OEIS} and~\cite[p.~7]{HP}).

The idea behind Conjecture~\ref{c:Wilf} is that \emph{P\'olya theory}
formulas (see e.g.~\cite{HP})
are fundamentally not Wilfian.  We should mention that we do not believe
the conjecture in view of Babai's recent quasipolynomial time algorithm for
{\sc Graph Isomorphism}~\cite{Bab}. While the connection is indirect,
it is in fact conceivable that both problems can be solved in \poly$(n)$ time.

\begin{question}
Let $\pi(n)$ denote the number of primes $\le n$.
Does $\{\pi(n)\}$ have a Wilfian formula of type~\WIV?
\end{question}

Initially, Wilf asked about formula of type~\WII~for $\{\pi(n)\}$,
and such algorithm was given in~\cite{LMO}.  Note that even the parity
of~$\pi(n)$ is hard to compute (cf.~\cite{TCH}), making unlikely a positive
answer to the question above.

\smallskip

\subsection{Complexity setting and graph enumeration}\label{ss:formula-comp}
Let $\cP_n$ denote the set of certain \emph{combinatorial objects} of size~$n$.
Formally, this means that one can decide if \ts $X\in \cP_n$ \ts in time~\poly$(n)$.
The problem of computing $a_n:=|\cP_n|$ has the input $n$ of \emph{bit-length} $O(\log n)$,
much too small for the (usual) polynomial hierarchy. Instead, the
\emph{exponential hierarchy} is used: $\NEXP$
for existence of combinatorial objects which can be verified in time $n^{O(1)}$,
and~\SEXP~for counting such objects.\footnote{To bring the problem
into the polynomial hierarchy, the input $n$ should be given
in \emph{unary}.}

For example, let $a_n=|\cP_n|$ be the set of (labeled) planar
$3$-regular $3$-connected graphs on~$n$ vertices.  Graphs
in $\cP_n$ are exactly graphs of simple $3$-dimensional polytopes.
Since testing each property can be done in \poly$(n)$ time,
the decision problem is naturally in~\NEXP, and the
counting problem is in~\SEXP.  In fact, the decision problem
is trivially in~\PP, since such graphs exist for all
even $n\ge 4$ and don't exist for odd~$n$.
Furthermore, Tutte's formula for the number of rooted plane
triangulations gives a simple product formula for $a_n$,
and thus can be computed in~\poly$(n)$ time, see~\cite[Ch.~10]{Tutte}.

On the one hand, counting the number of non-Hamiltonian
graphs in $\cP_n$ is not naturally in~\SEXP, since testing
non-Hamiltonicity is \coNP-complete in this case~\cite{GJT}.
On the other hand, the corresponding decision problem
(the existence of such graphs) is again in~\PP~by Tutte's
disproof of Tait's conjecture, see~\cite[Ch.~2]{Tutte}.

\smallskip

Note that {\sc Graph Isomorphism} is in~\PP~for trees, planar graphs
and graphs of bounded degree, see e.g.~\cite[$\S$6.2]{Bab-handbook}.
The discussion above suggests the following counterpart
of Wilf's Conjecture~\ref{c:Wilf}.

\begin{conj} \label{c:planar-poly}
Let $a_n$ be the number of unlabeled plane
triangulations with $n$ vertices, and let~$b_n$ be the number of
$3$-connected planar graphs with $n$ vertices.
Then $\{a_n\}$ and $\{b_n\}$ can be computed
in {\rm poly}$(n)$ time.
\end{conj}

We are very optimistic about this conjecture as for maps this is
already known~\cite{Fusy}.  For triangulations
there is some recent evidence in~\cite{KS}.

\begin{thm} \label{t:planar-poly}
Let  $a_n$ be the number of unlabeled trees with $n$ vertices.
Then $\{a_n\}$ can be computed in {\rm poly}$(n)$ time.
\end{thm}

\begin{proof}
Denote by $b_n$ the number of unlabeled rooted trees with $n$ vertices.  We have
\ts $b_n \le C_{n-1}$ \ts since the Catalan number $C_{n}$ is the number of plane trees
with $n+1$ vertices, so $\log b_n = O(n)$.  We also have:
%P\'olya's 1937 theorem implies:
% $$B(t) \, = \, t\cdot \exp \left(B(t) \. + \. \frac12 B(t)^2 \. + \. \frac13 B(t)^3 \. + \ldots \right),
% \quad \text{where} \ \, B(t) \, := \, \sum_{n=1}^\infty \. b_n \ts t^n\ts.
% $$
%
$$
b_{n+1} \, = \, \frac1n \, \sum_{k=1}^{n} \. b_{n-k+1} \Biggl[\sum_{d\ts{}|\ts{}k}  \. d\ts b_d\Biggr]\ts,
$$
see e.g.~\cite[$\S$5.6]{Finch} and \cite[\href{https://oeis.org/A000081}{A000081}]{OEIS}.	
Thus, $\{b_n\}$ can be computed in \poly$(n)$ time.
On the other hand:
$$
a_n \, = \, b_n \. - \. \frac12 \. \sum_{k=1}^{n-1} \. b_k\ts b_{n-k} \. + \. \left\{\aligned
& \frac12 \. b_{n/2}\ts, \ \ n \  \text{even}; \\
& \ \ 0\ts, \ \ \ \text{otherwise},
\endaligned\right.
$$
see e.g.~\cite[$\S$3.2]{HP} and~\cite[\href{https://oeis.org/A000055}{A000055}]{OEIS}.
This implies the result.
\end{proof}

\smallskip

Let $a_n$ denote the number of $3$-regular labeled graphs on $2n$ vertices.
The sequence $\{a_n\}$ can be computed in polynomial time via
the following recurrence relation, see~\cite{GJ}
and~\cite[\href{https://oeis.org/A002829}{A002829}]{OEIS}.

{\small
\begin{equation}\label{eq:3-regular}
\aligned
& 3(3n-7)(3n-4)\cdot a_n \ = \ 9(n-1)(2n-1)(3n-7)(3n^2 - 4n + 2)\cdot a_{n-1} \\
& \quad \. + \. (n-1)(2n-3)(2n-1)(108n^3 - 441n^2 + 501n - 104) \cdot a_{n-2} \\
& \quad\. + \.
2(n-2)(n-1)(2n-5)(2n-3)(2n-1)(3n-1)(9n^2 - 42n + 43)\cdot a_{n-3} \\
& \quad\. - \.
2(n-3)(n-2)(n-1)(2n-7)(2n-5)(2n-3)(2n-1)(3n-4)(3n-1)\cdot a_{n-4}
\endaligned
\end{equation}
}

\nin
It was shown in~\cite{Ges} that similar polynomial recurrences exist for the number
of $k$-regular labeled graphs, for every fixed $k\ge 1$.

\begin{conj} \label{c:planar-poly-regular}
Fix $k\ge 1$ and let $a_n$ be the number of unlabeled
$k$-regular graphs with $n$ vertices.  Then $\{a_n\}$
can be computed in \poly$(n)$ time.
\end{conj}

For $k=1,2$ the problem is elementary, but for $k=3$ is related to
enumeration of certain $2$-groups (cf.~\cite{Luks}).

\smallskip

Consider now the problem of computing the number $f(m,n)$ of triangulations
of an integer $[m\times n]$ grid (see Figure~\ref{f:santos}).  This problem
is a distant relative of Catalan numbers $C_n$ in~\eqref{eq:formula-Catalan}
which Euler proved counts the number of triangulations of a convex $(n+2)$-gon
(see~\cite{Sta-cat}), and is one of the large family of triangulation
problems (see~\cite{DRS}). Kaibel and Ziegler prove in~\cite{KZ} that $f(m,n)$
can be computed in \poly$(n)$ time for every fixed~$m$, but report that
their algorithm is expensive even for relatively small $m$ and~$n$
(see~\cite[\href{https://oeis.org/A082640}{A082640}]{OEIS}).

\begin{question}
Can $\{f(n,n)\}$ can be computed in \poly$(n)$ time?
\end{question}

{\small
\begin{figure}[hbt]
 \begin{center}
   \includegraphics[height=2.6cm]{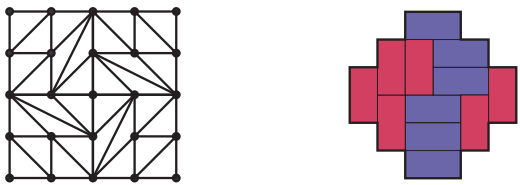}
   \caption{Grid triangulation of $[5 \times 5]$ and a domino tiling.}
   \label{f:santos}
 \end{center}
\end{figure}
}

\subsection{Computability setting and polyomino tilings}\label{ss:tilings}
Let $a_n$ be the number of domino tilings on a $[2n\times 2n]$ square.
The Kasteleyn and Temperley--Fisher classical determinant formula (1961)
for the number of perfect matchings of planar graphs gives a \poly$(n)$
time algorithm for computing $\{a_n\}$, see e.g.~\cite{Ken-dimer,LP-match}.
This foundational result opens the door to potential generalizations, but,
unfortunately, most of them turn out to be computationally hard.

First, one can ask about computing the number $b_n$ of $3$-dimensional
domino tilings of a $[2n\times 2n \times 2n]$ box.  Or how about
the seemingly simpler problem of counting the number $c_n$ of $3$-dimensional
domino tilings of a ``slim'' $[2\times n \times n]$ box?  We don't
know how to solve either problem, but both are likely to be
difficult.  The negative results include \SP-completeness
of the counting problem for general and slim regions~\cite{PY-domino,Val-alg},
and topological obstacles, see~\cite{FKMS} and~\cite[Prop.~8.1]{PY-domino}.

Now consider a fixed finite set $\bT=\{\tau_1,\ldots,\tau_k\}$
of general polyomino tiles on a square grid: $\tau_i \ssu \zz^2$, $1\le i\le k$.
To tile a region $\Ga \ssu \zz^2$, one must cover it with
copies of the tiles without overlap. These copies must be parallel
translations of $\tau_i$ (rotations and reflections are not allowed).
There exist $\NP$-complete tileability problems and~$\SP$-complete tiling
counting problems even for a fixed
set of a few small tiles.  We refer to~\cite{pak-horizons} for short
survey of the area.

\smallskip

For a fixed~$\bT$, let $g(m,n)$ denote the number of tilings of
$[m\times n]$ with~$\bT$.  Is $g(m,n)$ computable in polynomial
time?  The following conjecture suggests otherwise.

\begin{conj}\label{c:tileability-SEXP}
There exists a finite set of tiles $\bT$ such that counting
the number of tilings of \ts
$[n\times n]$ \ts with~$\bT$ is $\SEXP$-complete.
\end{conj}

In fact, until we started writing this survey, we always believed this
result to be known, only to realize that standard references such as~\cite{Boas}
fall a bit short.  Roughly, one needs to embed a $\SEXP$-complete language
into a counting tilings problem of a rectangle.  This is a
classical idea (see e.g.~\cite[$\S$5.3.4,~$\S$7.6.5]{MM}), which worked
well for many related problems. For example, the
\textsc{Rectangular Tileability} problem asks: given a finite set
of tiles $\bT$, do there exist integers $m$ and~$n$, such that
$\bT$ tiles $[m\times n]$?

\begin{thm}[Yang~\cite{Yang}]  \label{t:yang-rectangle}
The \textsc{Rectangular Tileability} problem is undecidable.
\end{thm}

In the proof, Yang embeds the \textsc{Halting Problem} into
\textsc{Rectangular Tileability}.  So can one embed a $\NEXP$-complete
problem into tileability of an \ts $[m\times n]$ \ts rectangle?  The answer is
\emph{yes} if $\bT$ is allowed to be part of the input.  In fact, even
Levin's original 1973 paper introducing \NP-completeness proposed
this approach~\cite{Lev}.  The following result should come as a
surprise, perhaps.

\begin{thm}[Lam--Miller--Pak, see~\cite{Lam}]  \label{t:LMP-rectangle}
Given $\bT$, the tileability of $[m\times n]$ can be decided in
$O(\log m + \log n)$ time.
\end{thm}

The proof is nonconstructive; it is based on \emph{Hilbert's Basis Theorem}
and the algebraic approach by F.~W.~Barnes.
A combination of theorems~\ref{t:yang-rectangle} and~\ref{t:LMP-rectangle}
shows that the constant implied by the $O(\cdot)$ notation is not
computable as a function of~$\bT$.  Roughly, we do know that a linear-time
algorithm exists, but given $\bT$ it is undecidable to find it.
Theorem~\ref{t:LMP-rectangle}
also explains why Conjecture~\ref{c:tileability-SEXP} remains open -- most
counting results in the area use parsimonious reductions (think bijections
between solutions of two problems), and in this case a different approach
is required.

\medskip

\section{Classes of combinatorial sequences}\label{ss:classes}

\subsection{Algebraic and D-algebraic approach}\label{ss:formula-classes}
Combinatorial sequences $\{a_n\}$ are traditionally classified depending
on the algebraic properties of their~GFs
$$A(t) \, = \, \sum_{n=0}^\infty \. a_n \ts t^n.
$$
We list here only four major classes:

\medskip

\nin
\qquad {\it \textbf{Rational}}: \qquad $A(t) =P(t)/Q(t)$, \ for some \. $P,\ts Q\in \zz[t]$,

\smallskip

\nin
\qquad {\it \textbf{Algebraic}}: \ \quad \.\ts $c_0 A^k \ts + \ts c_1 A^{k-1}\ts+\ts\ldots
\ts+ \ts c_k=0$, \ for some \, $k\in \nn$, \. $c_i\in \zz[t]$,

\smallskip

\nin
\qquad {\it \textbf{D-finite}}: \quad \quad \.\ts $c_0A \ts + \ts c_1A^\prime\ts + \ts\ldots\ts + \ts c_kA^{(k)}=b$,
\ for some \, $k\in \nn$, \. $b,\ts c_i\in \zz[t]$,

\smallskip

\nin
\qquad {\it \textbf{D-algebraic}}:  \ \. $Q\bigl(t,A,A,\ldots,A^{(k)}\bigr)=0$, \ for some \, $k\in \nn$, \. $Q\in \zz[t,x_0,x_1,\ldots,x_k]$.

\medskip

\nin
Here we exclude the trivial equation $0=0$.  \ts 
Note that rational~GFs are exactly those $\{a_n\}$ that satisfy a linear recurrence:
$$c_0 \ts a_n \, = \, c_1 \ts a_{n-1} \. + \. \ldots \. + \. c_k \ts a_{n-k}, \ \, \text{for some} \
k\in \nn, \. c_i\in \zz.
$$
Such sequences $\{a_n\}$ are called \emph{C-recursive} (or \emph{linearly recursive}).
For example, Fibonacci numbers satisfy~\eqref{eq:Fib1} and have GF \ts $(1-t-t^2)^{-1}$.
Similarly, \emph{Catalan numbers} have an algebraic GF by~\eqref{eq:formula-C}.   D-finite~GFs
(also called \emph{holonomic}) are exactly those $\{a_n\}$ that
satisfy polynomial recurrence
$$c_0(n) \ts a_n \, = \, c_1(n) \ts a_{n-1} \. + \. \ldots \. + \. c_k(n) \ts a_{n-k}, \ \, \text{for some} \
k\in \nn, c_i\in \zz[n]\ts.
$$
Such sequences $\{a_n\}$ are called \emph{P-recursive}.
Examples include $\{n!\}$, derangement numbers $\{D_n\}$ by~\eqref{eq:formula-Derangements},
the number of 3-regular graphs by~\eqref{eq:3-regular}, and the numbers $\{r_n\}$ of
involutions in~$S_n$, which satisfy \ts $r_{n} = r_{n-1} + (n-1)r_{n-2}$,
see~\cite[\href{https://oeis.org/A000085}{A000085}]{OEIS}.
Finally, {\it D-algebraic} GFs
(also called \emph{ADE} and \emph{hyperalgebraic})
include Euler numbers by equation~\eqref{eq:formula-E} and the 
number of partitions~$p(n)$, see below.

\begin{thm}[{see e.g.~\cite[Ch.~6]{Stanley-EC2}}]
$$\text{Rational \, $\ssu$ \, Algebraic \, $\ssu$ \, D-finite \, $\ssu$ \, D-algebraic. \hskip5.cm}
$$
\end{thm}

\nin
Here only the inclusion \ts {\it Algebraic  $\ssu$  D-finite} \ts
is nontrivial.  The following easy observation explains
the connection to the subject (we omit the proof).

\begin{prop}  \label{p:Wilf}
Sequences with \ts {\em D-algebraic} GFs have Wilfian formulas of type~\WI.
\end{prop}

In other words, if one wants to show that a sequence does not have a Wilfian
formula, then proving that it is {\it D-transcendental}, i.e.\ non-D-algebraic,
is a good start.\footnote{To simplify
exposition and for the lack of better terminology, here and in the future
we refer to sequences by the properties of their GFs.}
Unfortunately, even proving that a sequence is non-P-recursive
is often challenging (see below).

\begin{ex}[Bell numbers]{\rm  Let $B_n$ denote the number of
\emph{set partitions} of $\{1,\ldots,n\}$,
see~\cite{Stanley-EC2} and~\cite[\href{https://oeis.org/A000110}{A000110}]{OEIS}.
Let
$$y(t) \, =\, \sum_{n=0}^\infty \. \frac{B_n\ts t^n}{n!}\,, \qquad
z(t) \, =\, \sum_{n=0}^\infty \. B_n\ts t^n\.,
$$
be the exponential and ordinary GFs of Bell numbers, respectively.  On the one hand,
we have:
$$
y(t) \. = \. e^{e^t-1}, \qquad y'' \ts y \. - \. (y')^2 \. - \.  y'\ts y\.= \. 0\ts.
$$
Thus, $y(t)$ is D-algebraic, and the proposition implies that $\{B_n\}$  can be
computed in \poly$(n)$ time.  On the other hand, $z(t)$ is D-transcendental by
Klazar's theorem~\cite{Kla}.

This also implies that $y(t)$ is not D-finite.  Indeed, observe by definition,
that if a sequence $\{a_n\}$ is P-recursive, then so is $\{n!\ts a_n\}$,
which implies the result by taking $a_n=B_n/n!$ (cf.~\cite{LR}).  Of course,
there is a more direct way to prove that $y(t)$ is not $D$-finite by
repeated differentiation or via the asymptotics, see below.
This suggests the following advanced generalization of Klazar's theorem.

\begin{op}[Pak--Yeliussizov]
Suppose $\{a_n/n!\}$ is D-algebraic but not P-recursive.  Does this imply
that $\{a_n\}$ is D-transcendental?
\end{op}
}
\end{ex}

\smallskip

Before we proceed to more combinatorial examples, let us mention that
D-transcendental GFs are the subject of \emph{Differential Galois Theory},
which goes back to Liouville, Lie, Picard and Vessiot in the 19th century
(see e.g.~\cite{Ritt}), and continues to be developed~\cite{Put-S}.  Some
natural GFs are known to be D-transcendental, e.g.\ $\Ga(z)$, $\zeta(z)$, etc.,
but there are too few methods to prove this in most cases of interest.
Here are some of our favorite open problems along these lines,
unapproachable with existing tools.

\begin{conj}  \label{c:prime-trans}
$\sum_{n\ge 1} \ts p_n \ts t^n$ \ts and \ts $\sum_{n\ge 1} \ts \pi(n) \ts t^n$ \ts
are  D-transcendental.
\end{conj}

Here $p_n$ is the $n$-th prime, and $\pi(n)$ is the number of primes $\le n$,
as above.  Both GFs are known to be non-D-finite, as shown by
Flajolet, Gerhold and Salvy in~\cite{FGS} by asymptotic
arguments.  The authors quip: ``{\em Almost anything
is non-holonomic unless it is holonomic by design}''.
Well, maybe so. But the same applies for D-transcendence
where the gap between what we believe and what we can prove
is much wider.  The reader should think of such open problems as
the irrationality of $e+\pi$ and $\zeta(5)$, and imagine a similar
phenomenon in this case (cf.~\cite{KoZ}).
\begin{conj}
\ts ${\displaystyle{\sum_{n\ge 0} \ts \ts t^{n^3}}}$ \ts
is  D-transcendental.
\end{conj}

This problem should be compared with Jacobi's 1848 theorem that
the \emph{theta function}
\ts $\sum_{n}  \ts t^{n^2}$ \ts is D-algebraic~\cite{Jac}.
To understand
the difference, the conjecture is saying that there are no good
formulas governing the number of ways to write $n$ as a sum of
$k$~cubes, for any~$k$, the kind of formulas that exist for
sums of two, four and six squares, see~\cite[$\S$XX]{HW}.

\smallskip

\subsection{Combinatorial and asymptotic tools}
The following is a simple combinatorial criterion for
non-P-recursiveness.

\begin{thm}[\cite{GP-words}] \label{t:mod2}
Let $\{a_n\}$ be a P-recursive integer sequence. Consider the
infinite binary word \ts ${\bold w}=w_1w_2\cdots$ \ts defined by \ts
$w_n= a_n\,\,{\rm mod}\,\, 2$.  Then there exists a finite
binary word~$v$ that is not a subword of~$w$.
\end{thm}

Here by a \emph{subword} we mean a consecutive subsequence of letters.
For example, the infinite binary word for the sequence of
Fibonacci numbers \ts $\{F_n$~mod~$2\}$ \ts do not
contain~$(111)$, Catalan numbers \ts $\{C_n$~mod~$2\}$ \ts and derangement
numbers \ts $\{D_n$~mod~$2\}$ \ts do not contain~$(11)$, etc. In a different
direction, this implies that the \emph{binary Champernowne sequence}
(all natural numbers in binary, concatenated)
$$
\text{\bf \textcolor{blue}{0}, \. 1, \. \textcolor{blue}{0, \. 0}, \. 0, \. 1, \.
\textcolor{blue}{1, \. 0}, \. 1, \. 1, \. \textcolor{blue}{0, \. 0, \. 0}, \.
0, \. 0, \. 1, \. \textcolor{blue}{0, \. 1, \. 0}, \. 0, \. 1, \. 1, \ldots }
$$
is not P-recursive,
see~\cite[\href{https://oeis.org/A076478}{A076478}]{OEIS}.

Unfortunately, for many natural non-P-recursive sequences the assumption
in the theorem is much too strong.  For example, for the parity of partition numbers
\ts $\bigl\{p(n)$~mod~$2\bigr\}$, and for the odd primes modulo~4, \ts
$\bigl\{(p_n-1)/2$~mod~$2\bigr\}$, it is an open problem whether
all binary subwords appear (bet on yes).

\smallskip

The following  result is the best tool we have for proving
that a combinatorial sequence is not P-recursive.  Note that
deriving such asymptotics can be very difficult; we refer
to~\cite{FS,PW} for recent comprehensive monographs on the subject.

\begin{thm}\label{t:asy}
Let $\{a_n\}$ be a P-recursive sequence, s.t. $a_n \in \qqq$, \ts
$C_1^n < a_n < C_2^n$ for some $C_2>C_1>0$ and all $n \ge 1$.  Then
$$
a_n \, \sim \, \sum_{i=1}^m \. K_i \. \la_i^n \. n^{\al_i} \. (\log n)^{\be_i}\,,
$$
where \ts $K_i \in \rr_+$, \ts $\la_i \in\overline{\qqq}$, \ts $\al_i \in \qqq$,
and \ts $\be_i \in \nn$.
\end{thm}

The theorem is a combination of several known results~\cite{GP-words}.
Briefly, the generating series $\ca(t)$ is a $G$-function in a sense of Siegel~(1929),
which by the works of  Andr\'{e}, Bombieri, Chudnovsky, Dwork and Katz,
must satisfy an ODE which has only regular singular points and rational
exponents.  We then apply the Birkhoff--Trjitzinsky claim/theorem,
which in the regular case has a complete and self-contained proof
in~\cite{FS} (see Theorem~VII.10 and subsequent comments).

\begin{ex}[Euler numbers $E_n$]{\rm Recall that
$$E_n \. \sim \. \frac{4}{\pi}  \left(\frac{2}{\pi}\right)^n  n!
$$
(see e.g.~\cite[p.~7]{FS}).  Then $\{E_n\}$ is not P-recursive, since
otherwise \ts $E_n/n!\sim K\ts \la^N$ \ts with a transcendental base of
exponent \ts $\la = (2/\pi) \notin \ov{\qqq}$.  \ts Note that Euler numbers
can be computed in time \ts $O\bigl(n^{\frac43+\ep}\bigr)$~\cite{Har}. 
}\end{ex}

\begin{ex}[$n$-th prime $p_n$]{\rm Following~\cite{FGS}, recall that \ts
$p_n= n\log n + n \log \log n + O(n)$.  Observe that the \emph{harmonic number} $h_n$
is P-recursive by definition:
$$h_n \, = \, h_{n-1} \. + \. \frac1n  \, = \, 1 \. + \. \frac12 \. + \. \ts \ldots
\ts \. + \. \frac1n \, = \, \log n \. + \. O(1)\ts.
$$
Then $\{p_n\}$ is not P-recursive, since otherwise so is
$$p_n \, - \, n\ts h_n \, = \, n \log \log n + O(n),
$$
which is impossible by Theorem~\ref{t:asy}.
}\end{ex}

\smallskip

\subsection{Lattice walks} \label{ss:walks-lattice}
Let $\Ga=(V,E)$ be a graph
and let $v_0,v_1\in V$ be two fixed vertices.  Let $a_n$ be the
number of walks $v_0\to v_1$ in $\Ga$ of length~$n$.  This is a
good model which leads to many interesting sequences. For example,
Fibonacci number $F_n$ is the number of walks $1\to 1$ of
length~$n$ in the graph on $\{1,2\}$, with edges
$(1,1)$, $(1,2)$ and~$(2,1)$.

For general finite graphs we get C-recursive sequences $\{a_n\}$
with rational~GFs. For the usual walks $0\to 0$ on $\nn$ we get
Catalan numbers $a_{2n}=C_n$ as in~\eqref{eq:formula-Catalan},
while for $\pm 1$ walks in $\zz$ we get $a_{2n}=\binom{2n}{n}$,
both algebraic sequences.  Similarly, for $(0,\pm 1), (\pm 1,0)$
walks in $\zz^2$, we get $a_{2n}=\binom{2n}{n}^2$, which is
P-recursive but not algebraic~\cite{Furst}.  In higher
dimensions or for more complicated graphs, there is no
such neat formula.

\begin{thm} \label{t:walks-Zd}
Let $S\ssu \zz^d$ be a fixed finite set of steps, and let $a_n$
be the number of walks $O\to O$ in $\zz^d$ of length~$n$,
with steps in~$S$.  Then $\{a_n\}$ is P-recursive.
\end{thm}

This result is classical and follows easily from~\cite[$\S$6.3]{Stanley-EC2}.
It suggests that to obtain more interesting sequences one needs to look
elsewhere.  Notably, one can consider natural lattice walks on some
portion of~$\zz^d$.  There is a tremendous number of results in the
literature, remarkable both in scope and beauty.

In recent years, M.~Bousquet-M\'elou and her coauthors initiated a
broad study of the subject, and now have classified all walks in the
first quadrant which start and end at the origin~$O$, and have a fixed
set~$S$ of steps with both coordinates in $\{0,\pm1\}$.  There are in
principle $2^8-1=255$ such walks, but some of them are trivial and some
are the same up to symmetries.  After the classification was completed,
some resulting sequences were proved algebraic
(say, \emph{Kreweras walks} and \emph{Gessel walks});
very surprisingly so, some are D-finite (not a surprise given
Theorem~\ref{t:walks-Zd}), some are D-algebraic (this required
development of new tools), and some are D-transcendental
(it is amazing that this can be done at all).

\begin{ex}[Case~16]\label{ex:case-16}
{\rm
Let \ts $S=\bigl\{(1,1), \ts (-1,-1), \ts (-1,0), \ts (0,-1)\bigr\}$,
and let $a_n$ be the number of walks $O\to O$ in the first quadrant
of length~$n$, with steps in~$S$,
see~\cite[\href{https://oeis.org/A151353}{A151353}]{OEIS}.
It was shown in~\cite[Case~16]{BRS} that
$$
a_n \, \sim \, K\ts \la^n \ts n^\al\ts,
$$
where \ts $\la \approx 3.799605$ \ts is a root of
\. $x^4 +x^3 -8x^2 -36x -11=0$, and \ts $\al \approx -2.318862$ \ts
satisfies $c=-\cos(\pi/\al)$, and $c$ is a root of
$$
y^4 \ts- \ts \frac92 \ts y^3 \ts + \ts \frac{27}{4} y^2 - \frac{35}{8}\ts y + \frac{17}{16} \, = \. 0\ts.
$$
Since $\al \notin \qqq$, Theorem~\ref{t:asy} implies that
$\{a_n\}$ is not P-recursive.
}\end{ex}

% It would be interesting to obtain an elementary proof of this result given
% the symmetry of the steps.  Such elementary proof would perhaps be
% applicable to other settings.

We refer to~\cite{Bou,BM} for a comprehensive overview of
the background and early stages of this far-reaching project,
and to~\cite{BBR,BMKS} for some recent developments which are
gateways to references.  Finally, let us mention a remarkable recent
development~\cite{DHRS}, which proves D-transcendence for many
families of lattice walks.  Let us single out just one of the many
results in that paper:

\begin{thm}[{\cite[Thm.~5.8]{DHRS}}]  Sequence
$\{a_n\}$ defined in Example~\ref{ex:case-16}
is D-transcendental.
\end{thm}

\smallskip

In conclusion, let us mention that $\{a_n\}$ can be
computed in polynomial time straight from definition using
dynamic programming, since the number of points reachable
after~$n$ steps is \poly$(n)$.  This leads us to consider
walks with constraints or graphs of superpolynomial growth.

\begin{conj}[cf.~\cite{Zei}]
Let $a_n$ denotes the number of self-avoiding walks $O\to O$
in $\zz^2$ of length~$n$.  Then the sequence $\{a_n\}$ has
no Wilfian formula of type~\WI.
\end{conj}

We refer to~\cite{Gut} for an extensive investigation of
\emph{self-avoiding walks} and its relatives,
and a review of the literature.

\smallskip

\subsection{Walks on Cayley graphs}\label{ss:walks-Cayley}
Let $G=\<S\>$ be a finitely generated group~$G$
with a % symmetric
generating set $S$. Let $a_n=a_n(G,S)$ be the number of words
in~$S$ of length~$n$ equal to~$1$; equivalently, the number of
walks $1\to 1$ of length~$n$, in the Cayley graph $\Ga=\Ga(G,S)$.
In this case $\{a_n\}$ is called the \emph{cogrowth sequence} and
its GF $A(t)$ the \emph{cogrowth series}.  They were introduced
by P\'olya in~1921 in the probabilistic context of random walks on graphs,
and by Kesten in the context of amenability~\cite{Kes}.

The cogrowth sequence $\{a_n\}$ is C-recursive if only if $G$ is finite~\cite{Kuk1}.
It is algebraic for the infinite dihedral group~\cite{Hum}, for the free
 group~\cite{Hai} and for free products of finite groups~\cite{Kuk2},
all with standard generators.  The cogrowth sequence is P-recursive
for many abelian groups~\cite{Hum}, and for the
\emph{Baumslag-Solitar groups} $G=\BS(k,k)$ in the standard presentation
\ts $\BS(k,\ell) = \<x,y\.|\.x^ky=yx^\ell\>$, see~\cite{E+}.

\begin{thm}[\cite{GP-words}]  \label{t:prob}
The sequence \ts $\bigl\{a_n(G,S)\bigr\}$ \ts
is not P-recursive for all \emph{symmetric} $S=S^{-1}$,
and the following classes of groups~$G$:\\
$(1)$ \ts virtually solvable groups of exponential growth with finite Pr\"ufer rank; \\
$(2)$ \ts amenable linear groups of superpolynomial growth; \\
$(3)$ \ts groups of weakly exponential growth
$$
A \ts e^{n^{\al}} \. < \ga_{G,S}(n) \. < \. B \ts e^{n^{\be}}\ts,
$$
where $A, B>0$, and $0< \al,\be<1$; \\
$(4)$ \ts the Baumslag--Solitar groups $\BS(k,1)$, where $k\ge 2$; \\
$(5)$ \ts the lamplighter groups $L(d,H)=H \wr\zz^d$, where~$H$
is a finite abelian group and $d\ge 1$.
\end{thm}

Since  $G \simeq \zz \ltimes\zz^2$ with a free action of~$\zz$,
is linear of exponential growth, by~$(2)$ we obtain a solution
to the question originally asked by Kontsevich, see~\cite{Sta-mfo}.

\begin{cor}[\cite{GP-words}]  \label{c:konts}
There is a linear group $G$ and a symmetric generating set~$S$,
s.t.\ the sequence \ts $\bigl\{a_n(G,S)\bigr\}$ \ts
is not P-recursive.
\end{cor}

The proof in~\cite{GP-words} is a combination of many results by
different authors.  For example, for $G=\BS(k,1)$, $k\ge 2$,
and every symmetric $\<S\>=G$,
there exist $C_1,C_2>0$ that depend on~$S$, s.t.
\begin{equation}\label{eq:woess}
|S|^n \. e^{-C_1\ts \sqrt[3]{n}} \, \le \, a_n(G,S) \, \le \, |S|^n \. e^{-C_2\ts \sqrt[3]{n}}\.,
\end{equation}
see~\cite[$\S$15.C]{Woe}.  The result now follows from Theorem~\ref{t:asy}.

It may seem from Theorem~\ref{t:prob} that the properties of
\ts $\{a_n(G,S)\}$ \ts depend only on~$G$, but that is false.
In fact, for $G=F_k\times F_\ell$ there are generating sets
with both P-recursive and non-P-recursive sequences; the negative
result in this case is given is proved in~\cite{GP-words}
by using Theorem~\ref{t:mod2}.
% , but the latter set is non-symmetric.
For groups in Theorem~\ref{t:prob}, the result is a byproduct of
probabilistic tools used in establishing the
asymptotics such as~\eqref{eq:woess}.  In fact, the probabilities
of return of the random walk \ts $a_n(G,S)/|S|^n$ \ts always have
the same growth under \emph{quasi-isometry}, see
e.g.~\cite{Woe}.\footnote{While the leading term in the asymptotics
remains the same, lower order terms can change for different~$S$,
see~\cite[$\S$17.B]{Woe}.}

It is unlikely that any of the sequences in the theorem are D-algebraic,
but we really have no idea nor any tools to establish such a result
other than by a direct calculation.  An exact asymptotic result
is known for a particular walk on the lamplighter group is given
in~\cite{Rev}.  Let \ts $G=L(1,\zz_2)=\zz_2\wr \zz$, with a symmetric
generating (multi-)set \ts $S=S_1S_2S_1$, where $S_1=\{(1,0), e\}$
and $S_2=\{(0,\pm 1)\}$.
In other words, each generator is a sequence of moves: turn the lamp on or off,
make a step left or right, then turn the lamp on or off at the new location,
all with probability~$1/2$.  Then:
$$
a_n(G,S) \, = \, K \. n^{1/6} \. e^{-C\ts \sqrt[3]{n}}, \quad  \text{where} \ \
K \. = \. \frac{2^{2/3}\ts \pi^{5/6}}{3^{1/2} \ts (\log 2)^{2/3}}\,, \ \
C\. = \. 3\cdot 2^{1/3} \ts \bigl(\pi \ts \log 2\bigr)^{2/3}\..
$$
It would be interesting to see if this sequence is D-algebraic.

In a forthcoming paper~\cite{GP-ADE}
we construct an explicit but highly artificial non-symmetric set \ts
$S\ssu F_k\times F_\ell$ \ts with D-transcendental cogrowth sequence.
In~\cite{KP2} we use the tools in~\cite{KP1} to prove that
groups have an uncountable set of \emph{spectral radii} \ts
$$
\rho(G,S) \, := \, \lim_{n\to \infty} a_n(G,S)^{1/n}\ts.
$$
Since the set of D-algebraic sequence is countable, this implies the
existence of D-transcendental Cayley graphs with symmetric~$S$,
but such a proof is nonconstructive.

\begin{op}\label{op:prob}
Find an explicit construction of $\Ga(G,S)$ when $S$ is symmetric,
and \ts $\{a_n(G,S)\}$ \ts is D-transcendental.
\end{op}

The sequences $\{a_n\}$ have been computed in very few special cases. For
example, for $\PSL(2,\zz)=\zz_2 \ast \zz_3$ with the natural
symmetric generating set, the cogrowth series $A(t)$ is computed
in~\cite{Kuk2}:
{\small
$$\aligned
&A(t) \, = \, \frac{(1+t)\bigl(-t + t^2-8t^3+3t^4-9t^5 +
(2-t+6t^2)\sqrt{\cR(t)}\bigr)}{2(1-3t)(1+3t^2)(1+3t+3t^2)(1-t+3t^2)}\., \\
& \ \  \text{where} \ \ \cR(t)\. = \. 1-2t+t^2-6t^3-8t^4-18t^5+9t^6-54t^7+81t^8.
\endaligned
$$ }
%\nin
There are more questions than answers here.  For example,
can cogrowth sequence be computed for nilpotent groups?

Before we conclude, let us note that everywhere above we
are implicitly assuming that $G$ either has a faithful
rational representation, e.g.\ $G=\BS(k,1)$ as in~$(4)$ above,
or more generally has the \emph{word problem} solvable in
polynomial time (cf.~\cite{LZ}). The examples include the
\emph{Grigorchuk group}~$\gg$, which is an example of~$3$,
see~\cite{GP} and the lamplighter groups $L(d,H)$ as in~$(5)$.
Note that in general the word problem can be superpolynomial or
even unsolvable, see e.g.~\cite{Mil}, in which case $\{a_n\}$
is no longer a combinatorial sequence.

\smallskip

\subsection{Partitions} \label{ss:formula-part}  Let $p(n)$ be the
number of integer partitions of~$n$, as in~\eqref{eq:formula-part}.
We have the \emph{Hardy--Ramanujan formula}:
\begin{equation}\label{eq:HR-asy}
p(n) \. \sim  \. \frac{1}{4\ts n \ts
\sqrt{3}} \, e^{\pi\ts \sqrt{\frac{2n}{3}}} \quad \text{as} \ \ n\to \infty.
\end{equation}
(see e.g.~\cite[VIII.6]{FS}).  Theorem~\ref{t:asy} implies that $\{p(n)\}$
is not P-recursive.  On the other hand, it is known that \ts
$$
F(t) \, := \, \sum_{n=0}^\infty \. p(n)\ts t^n \, = \, \prod_{i=1}^\infty \frac{1}{1-t^i}
$$
satisfies the following ADE:\footnote{This equation was found by Martin Rubey, see
\ts \href{https://mathoverflow.net/questions/47611/exact-formulas-for-the-partition-function/47706}{https://tinyurl.com/y7ewapjc}.}
%{\small
\begin{equation}\label{eq:part-ADE}
\aligned
4 \ts F^3 \ts F'' \. &  + \. 5\ts t\ts F^3 \ts F''' \. + \.  t^2 \ts F^3 \ts F^{(4)} \. - \.
16 \ts F^2 \ts (F')^2 \. - \. 15 \ts t \ts F^2\ts F'\ts F'' \. - \. 39 \ts t^2 \ts F^2 \ts (F'')^2 \\
% & \hskip1.25cm
& + \.  20 \ts t^2 \ts F^2 \ts F' \ts F''' \. + \. 10 \ts t \ts F\ts (F')^3 \. + \.
12 \ts t^2 \ts F\ts (F')^2 \ts F'' \. + \. 6\ts t^2 \ts (F')^4 \, = \, 0\ts.
\endaligned
\end{equation}
(cf.~\cite{MiC,Zag}). A quantitative version of Proposition~\ref{p:Wilf} then
implies that $p(n)$ can be computed in time $O(n^{4.5+\ep})$, for all $\ep>0$.
For comparison, the dynamic programming takes
$O(n^{2.5})$ time, where $O(\sqrt{n})$ comes
as the cost of addition.  Similarly, \emph{Euler's recurrence} famously
used by MacMahon~(1915) to compute~$p(200)$, gives an $O(n^2)$ algorithm:
\begin{equation}
\label{eq:part-euler-recurrence}
p(n) \. = \. p(n-1) \ts + \ts p(n-2)\ts - \ts  p(n-5) \ts - \ts  p(n-7)
\ts + \ts  p(n-12)\ts + \ts  p(n-15) \ts - \ts  \ldots
\end{equation}
(cf.~\cite{C+}).  
Finally, there is nearly optimal \ts $O\bigl(\sqrt{n} (\log n)^{4+\ep}\bigr)$ \ts 
time algorithm given in~\cite{Joh}. It is based on the Hardy--Ramanujan--Rademacher 
sharp asymptotic formula, which extends~\eqref{eq:HR-asy} to $o(1)$ additive error.  

Now, for a subset \ts $\ca \subseteq \{1,2,\ldots\}$, let $p_\ca(n)$ denote
the number of partitions of $n$ into parts in~$\ca$.  The dynamic programming
algorithm is easy to generalize to every \ts $\{p_\ca(n)\}$ \ts where the
membership \ts $a\in? \ts\ca$ \ts can be decided in \poly$(\log a)$ time, giving
a Wilfian formula of type~$\WI$.  This is polynomially optimal for
partitions into primes~\cite[\href{https://oeis.org/A000607}{A000607}]{OEIS}
 or squares~\cite[\href{https://oeis.org/A001156}{A001156}]{OEIS}, but not
 for sparse sequences.

\begin{prop}  \label{p:exp-part}
Let \ts $\ca=\{a_1,a_2,\ldots\}$, such that \ts
$a_k \ge c^k$, \ts for some $c>1$ and all $k\ge 1$.
Then \ts $p_\ca(n)= n^{O(\log n)}$.
\end{prop}

Thus, $p_\ca(n)$ as in the proposition could in principle have a Wilfian
formula of type~\WIII.  Notable examples include the number $q(n)$ of
\emph{binary partitions} (partitions of $n$ into powers of~$2$),
see \cite[\href{https://oeis.org/A000123}{A000123}]{OEIS},
% \emph{partitions into factorials} \cite[\href{https://oeis.org/A064986}{A064986}]{OEIS},
\emph{partitions into Fibonacci numbers} \cite[\href{https://oeis.org/A003107}{A003107}]{OEIS},
and \emph{\textsf{s}-partitions} defined as partitions into $\{1,3,7,\ldots,2^k-1,\ldots\}$
\cite[\href{https://oeis.org/A000929}{A000929}]{OEIS}.

\begin{thm}[\cite{PY-binary}]
Let \ts $\ca=\{a_1,a_2,\ldots\}$, and suppose \ts $a_{k}/a_{k-1}$ \ts
is an integer $\ge 2$, for all $k> 1$.  Suppose also that membership $x\in \ca$
can be decided in \ts{\rm poly}$(\log x)$ time.  Then \ts $\{p_\ca(n)\}$
can be computed in time \ts{\rm poly}$(\log n)$.
\end{thm}

This covers binary partitions, partitions into factorials
\cite[\href{https://oeis.org/A064986}{A064986}]{OEIS}, etc.

\smallskip

We conjecture that partitions into Fibonacci numbers and
\textsf{s}-partitions also have Wilfian formulas of
type~\WIII. Cf.~\cite{Rob} for an algorithm for partitions
into distinct Fibonacci numbers.  Note also that membership
can be tested in polynomial time: \ts $N$ is a Fibonacci number
if and only if $5N^2+4$ or $5N^2-4$ is a perfect square~\cite{Ges-Fib}.

Other partition sequences \ts $\{p_\ca(n)\}$ \ts with $\cA$ as
in Proposition~\ref{p:exp-part}, could prove less tractable.
These include \emph{partitions into Catalan
numbers} \cite[\href{https://oeis.org/A033552}{A033552}]{OEIS}
and \emph{partitions into partition numbers}
\cite[\href{https://oeis.org/A007279}{A007279}]{OEIS}.

We should mention that connection between algebraic properties
of GFs and complexity goes only one way:

\begin{thm}
The sequence $\{q(n)\}$ of the number of binary partitions
is D-transcendental.
\end{thm}

This follows from the \emph{Mahler equation}
$$
Q(t) \ts - \ts t \ts Q(t) \ts - \ts Q(t^2) \. = \. 0\ts, \quad
\text{where} \quad Q(t) \. = \. \sum_{n=0}^\infty \ts q(n) \ts t^n,
$$
see e.g.~\cite{DHR}. We conjecture that $\{a_n\}$ and $\{b_n\}$
from Conjecture~\ref{c:planar-poly} satisfy similar functional
equations, and are also D-transcendental.

\smallskip

\subsection{Pattern avoidance} \label{ss:formula-avoid}
Let $\si \in S_n$ and $\om \in S_k$.  The permutation $\si$ is said
to \emph{contain} the \emph{pattern} $\om$ if there is a subset
$X\subseteq \{1,\ldots,n\}$, $|X|=k$, such that $\si|_X$ has
the same relative order as~$\om$.  Otherwise, $\si$
is said to \emph{avoid}~$\om$.

Fix a set of patterns $\cf \subset S_k$. Let $\Av_n(\mathcal{F})$ denote
the number of permutations $\si\in S_n$ \emph{avoiding}
all patterns $\om \in\cf$.  The sequence \ts $\{\Av_n(\cf)\}$ \ts
is the fundamental object of study in the area of \emph{pattern avoidance},
extensively analyzed from analytic, asymptotic and combinatorial
points of view.

The subject was initiated by MacMahon (1915) and Knuth (1973),
who showed that \ts $\Av_n(123) =\Av_n(213) = C_n$, the
$n$-th Catalan number~\eqref{eq:formula-Catalan}.  The
\emph{Erd\H{o}s--Szekeres theorem} (1935) on longest increasing
and decreasing subsequences in a permutation can also be phrased
in this language:  \ts
$\Av_n(12\cdots k, \ell\cdots 21) \ts = \ts 0$, for all \ts
$n>(k-1)(\ell-1)$.

To give a flavor of subsequent developments, let us mention a
few more of our most favorite results.  Simion--Schmidt (1985)
proved \ts $\Av_n(123,132,213)=F_{n+1}$, the Fibonacci numbers.
Similarly, Shapiro--Stephens (1991) proved \ts $\Av_n(2413,3142)=S(n)$,
the Schr\"oder numbers \cite[\href{https://oeis.org/A006318}{A006318}]{OEIS}.
The celebrated Marcus--Tardos theorem~\cite{MT} states
that \ts $\{\Av_n(\om)\}$ \ts is at most exponential, for all $\om\in S_k$,
with a large base of exponent for random $\om\in S_k$~\cite{Fox}.
We refer to~\cite{Kit,Kla-survey,Vat} for many results on the
subject, history and background.

\smallskip

The \emph{Noonan--Zeilberger conjecture}~\cite{NZ}, first posed as a
question by Gessel~\cite{Ges}, states that the sequence
\ts $\{\Av_n(\cf)\}$ \ts is P-recursive for all $\cF\ssu S_k$.
It was recently disproved:

\begin{thm}[\cite{GP-NZ}]
There is \ts $\cF\ssu S_{80}$, \ts $|\cF|<30,000$, such that
\ts $\{\Av_n(\cf)\}$ \ts is not P-recursive.
\end{thm}

We extend this result in a forthcoming paper~\cite{GP-ADE}, where
we construct a D-transcendent sequence \ts $\{\Av_n(\cf)\}$,
for some $\cf\ssu S_{80}$.  Both proofs involve
embedding of Turing machines into the problem modulo~2. We also
prove the following result on complexity of
counting pattern-avoiding permutations,
our only result forbidding Wilfian formulas:

\begin{thm}[\cite{GP-NZ}]
If \. $\EXP\ne \oplus\EXP$, then \ts $\Av_n(\cf)$~{\rm mod}~$2$
\ts cannot be computed in \textrm{\poly}$(n)$ time.
\end{thm}

Here~$\oplus\EXP$ is the class of counting modulo~2 problems
of combinatorial objects in~$\NEXP$.
In other words, computing parity of the number of
pattern-avoiding permutations is likely hard.
We conjecture that $\Av_n(\cf)$ is $\SEXP$-complete,
but we are not very close to proving this.

\begin{thm}[\cite{GP-NZ}]
The problem whether \ts $\Av_n(\cf)=\Av_n(\cf') \mod~2$ \ts for all~$n$,
is undecidable.
\end{thm}

The theorem implies that in some cases even a large amount of computational
evidence in pattern avoidance is misleading.  For example, there exists two
sets of patterns \ts $\cf,\cf'\in S_k$, so that the first time
they have different parity is for $n>$ tower of 2s of height~$2^k$.

\smallskip

Finally, let us mention an ongoing effort to find a small set of
patterns~$\cF$, so that \ts $\{\Av_n(\cf)\}$ \ts is not P-recursive.
Is one permutation enough?  It is known that $\{\Av_n(1342)\}$
is algebraic~\cite{Bona}, while $\{\Av_n(1234) \}$ is
P-recursive~\cite{Ges}.  One of the most challenging problems is to analyze
\ts $\{\Av_n(1324)\}$, the only 4-pattern remaining.  The asymptotics
obtained experimentally in~\cite{CGZ} based on the values for $n\le 50$, suggests:
\[
\Av_n(1324) \. \sim \. B\ts \la^n \ts \mu^{\sqrt{n}} \ts n^\al,
\]
where $\la=11.600 \pm 0.003$, $\mu = 0.0400 \pm 0.0005$, $\al = -1.1 \pm 0.1$.
By Theorem~\ref{t:asy}, this is a convincing evidence against \ts $\{\Av_n(1324)\}$ \ts
being P-recursive.  While proving such a result remains out of reach,
the following problem could be easier.

\begin{op}\label{op:1324}
Can \ts $\{\Av_n(1324)\}$ \ts be computed in \textrm{poly}$(n)$
time?\footnote{In 2005, Doron Zeilberger expressed doubts that \ts
$\Av_{1000}(1324)$ \ts could  be computed even by Hashem.
This sentiment has been roundly criticized on both mathematical
and theological grounds (see~\cite{Ste}). }  More generally,
can one find a single permutation~$\pi$ such that
\ts $\{\Av_n(\pi)\}$ \ts cannot be computed in \text{poly}$(n)$
time?  Is the computation of \ts $\{\Av_n(\pi)\}$ \ts easier
or harder for \emph{random} permutations $\pi \in S_k$?
\end{op}

In the opposite direction, let us mention a sequence \ts
$\bigl\{\Av_n(4123, 4231, 4312)\bigr\}$ \ts which does have a
Wilfian formula of type~$\WI$, with an extremely strong computational
evidence for being D-transcendental~\cite{A++}.

\medskip

\section{Bijections}\label{s:bij}

\subsection{Counting and sampling via bijections} \label{ss:bij-enum}
There is an ocean of bijections between various combinatorial
objects.  They have a variety of uses: to establish a theorem,
to obtain refined counting, to simplify the proof, to make the
proof amenable for generalizations, etc.  Last but not least,
an especially beautiful bijection is often viewed as
a piece of art, an achievement in its own right, a result
to be taught and admired.

From the point of view of this survey, bijections \ts
$\vp: \cA_n \to \cB_n$ \ts are simply algorithms which
require complexity analysis.  There are two standard
applications of such bijections.  First, their existence
allows us to reduce counting of $\bigl\{|\ca_n|\bigr\}$
to counting of $\bigl\{|\cb_n|\bigr\}$.  For example, the
classical \emph{Pr\"ufer algorithm} allows counting of
spanning trees in $K_n$, reducing it to Cayley's
formula~\eqref{eq:formula-Cayley}.

Second and more recent application is to \emph{random sampling}
of combinatorial objects.  Oftentimes, one of the sets has a
much simpler structure which allows (nearly) uniform sampling.
To compare the resulting algorithm with other competing approaches
one then needs a worst case or average-case analysis
of the complexity of the bijection.

Of course, most bijections in the literature are so straightforward that
their analysis is elementary, think of the Pr\"ufer algorithm or
the classical ``plane trees into binary trees'' bijection~\cite{dBM}.
But this is also what makes them efficient.  For example, the bijections
for planar maps are amazing in their elegance, and have some
important applications to statistical physics; we refer to~\cite{Scha}
for an extensive recent survey and numerous references.

% Below we mention a few notable examples where analysis becomes
% rather difficult.

Finally, we should mention a number of \emph{perfect sampling}
algorithms, some of which in the right light can also be viewed
as bijections.  These include most notably general techniques
such as \emph{Boltzmann samplers}~\cite{DFLS} (see also~\cite{AD,BFKV}),
and \emph{coupling from the past}~\cite{LPW}.
Note also two beautiful ad hoc algorithms: \emph{Wilson's LERW}~\cite{Wil}
and the \emph{Aldous--Broder algorithm} for sampling uniform spanning trees
in a graph (both of which are highly nontrivial already for $K_n$),
see e.g.~\cite{LPW}.

% It is worth looking at a few of them before we proceed to more
% theoretical questions.

\smallskip

\subsection{Partition bijections} \label{ss:bij-complexity}
This is a large subject in its own right, with many results and
open problems.  For example, the \emph{Bressoud--Zeilberger
involution}~\cite{BZ} proves Euler's recurrence~\eqref{eq:part-euler-recurrence}.
At the same time, the equation implied by the ADE recurrence~\eqref{eq:part-ADE}
does not yet have a combinatorial proof, and looking for such a proof would
not be advisable.  We refer to~\cite{Pak-psurvey} for an extensive survey.

Let $q(n)$ denote the number of \emph{concave partitions} defined by \ts
$\la_i-\la_{i+1} \ge \la_{i+1} - \la_{i+2}$ for all~$i$.
Then $\{q(n)\}$ can be computed in \poly$(n)$ time.
To see this, recall \emph{Corteel's bijection} between
convex partitions and partitions into triangular numbers
\cite[\href{https://oeis.org/A007294}{A007294}]{OEIS}.  We then have:
$$
\sum_{n=1}^\infty \. q(n) \ts t^n \, = \, \prod_{k=2}^\infty \.
\frac{1}{1-t^{\binom{k}{2}}}\,,
$$
see~\cite{CCH}. This bijection can be described as a linear
transformation that can be computed in polynomial time~\cite{CS,Pak-geom}.
More importantly, the bijections allow random sampling of
concave partitions, leading to their limit shape~\cite{CCH,DP2}.

On the opposite extreme, there is a similar
\emph{Hickerson bijection} between \textsf{s}-partitions
and partitions with $\la_i\ge 2 \la_{i+1}$ for all $i\ge 1$,
see~\cite{CCH,Pak-geom}.  Thus, both sets are equally hard
to count, but somehow this makes the problem more interesting.

The Garsia--Milne celebrated \emph{involution principle}~\cite{GM} combines
the Schur and Sylvester bijections in an iterative manner,
giving a rather complicated bijective proof of the \emph{Rogers--Ramanujan
identity}:
\begin{equation}\label{eq:RR-id}
1\. + \. \sum_{k=1}^\infty \frac{t^{k^2}}{(1-t)(1-t^2)\cdots (1-t^k)}
\, \. = \. \,
\prod_{i=0}^\infty \. \frac{1}{(1-t^{5i+1})(1-t^{5i+4})}\..
\end{equation}
To be precise, they constructed a bijection $\Psi_n:\cP_n \to \cQ_n$,
where $\cP$ is the set of partitions into parts $\la_i \ge \la_{i+1}+2$,
and $\cQ$ is the set of partitions into parts $\pm 1$~mod~$5$.
In~\cite[$\S$8.4.5]{Pak-psurvey} we conjecture that $\Psi_n$ requires \ts
$\exp n^{\Omega(1)}$ \ts iterations in the worst case.  Partial
evidence in favor of this conjecture is our analysis of O'Hara's
bijection in~\cite{Konv-Pak}, with an \ts $\exp\Omega(\sqrt[3]{n})$ \ts
worst-case lower bound.  On the other hand, the iterative proof
in~\cite{BouP} for~\eqref{eq:RR-id} requires only $O(n)$ iterations.

\smallskip

\subsection{Plane partitions and Young tableaux} \label{ss:bij-solid}
Denote by $pp(n)$ the number of \emph{plane} (also called \emph{solid})
partitions.  MacMahon famously proved in 1912 that
$$
\sum_{n=0}^\infty \. pp(n) \ts t^n \, = \,
\prod_{k=1}^\infty \frac{1}{(1-t^k)^k}\.,
$$
which gives a \poly$(n)$ time algorithm for computing $sp(n)$.
This identity follows from a variation on the classical
\emph{Hillman-Grassl} and \emph{RSK} bijections,
see e.g.~\cite[$\S$9.1]{Pak-psurvey}.
Application to sampling of this bijection have been analyzed in~\cite{BFP}.
On the other hand, there is strong evidence that the RSK-based algorithms
cannot be improved.  While we are far from proving this, let us note that
in~\cite{PV} we show linear-time reductions between all major bijections
in the area, so a speedup of one of them implies a speedup of all.

The remarkable \emph{Krattenthaler bijection} allows enumerations of
solid partitions that fit into \ts $[n\times n\times n]$ \ts box~\cite{Kratt}.
This bijection is based on top of the \emph{NPS algorithm}, which has also
been recently analyzed~\cite{NS,SS}. Curiously, there are no analogous
results in $d\ge 4$ dimensions, making counting such $d$-dimensional
partitions an open problem (cf.~\cite{Gov}).

\smallskip

\subsection{Complexity of bijections} \label{ss:bij-exist}
Let us now discuss the questions~$(2')$ in the introduction,
about the nature of bijections \ts $\vp:\cP_n\to \cQ_n$ \ts from
an algorithmic point of view.

Let \ts $|\cP_n|=|\cQ_n|$ and think of $\vp$ as a map.
We require both $\vp$ and $\vp^{-1}$ to be computable in
polynomial time.\footnote{Here we are trying to avoid having
\emph{one-way functions}, which play an important role in
cryptography, but are distracting in this setting.}
If that's all we want, it suffices to
show that $\cP_n$ and~$\cQ_n$ can be enumerated in polynomial
time.  Here by \emph{enumerated} we mean a bijection
$\phi: \cP_n \to \{1,\ldots,|\cP_n|\}$, where both $\phi$
and $\phi^{-1}$ are computable in polynomial time.

For example, the dynamic programming plus \emph{divide-and-conquer}
proves that the sets of partitions $\cP_n$ and $\cQ_n$
on both sides of the Rogers--Ramanujan identity~\eqref{eq:RR-id},
can be enumerated in \poly$(n)$ time.  This gives a bijection \ts
$\vp_n: \cP_n \to \cQ_n$, proving the identity. But since proving
validity of such construction would require prior knowledge of
$|\cP_n|=|\cQ_n|$, from a combinatorial point of view this
bijection is unsatisfactory.

Alternatively, one can think of a bijection as an
algorithm that computes a given map~$\vp_n$ as above in
\poly$(n)$ time.  This is a particularly unfriendly setting,
as one would essentially need to prove new \emph{lower bounds}
in complexity.  Worse, we proved in~\cite{Konv-Pak} that in
some cases O'Hara's algorithm requires superpolynomial time,
while the map given by the algorithm can be computed in
\poly$(n)$ time using \emph{integer programming}.
Since this is the only nice bijective proof of the
\emph{Andrews identities} that we know (see~\cite{Pak-psurvey}),
this suggests that either we don't understand the nature
of these identities or have a very restrictive view of what
constitutes a combinatorial bijection. Or, perhaps, the
complexity approach is simply inapplicable in this
combinatorial setting.

There are other cases of unquestionably successful
bijections which are inferior to other algorithms from the
complexity point of view.  For example, stretching the
definitions a bit, Wilson's LERW algorithm for generating
random (directed) spanning trees uses
exponential time on directed graphs~\cite{Wil}, while a
straightforward algorithm based on the \emph{matrix-tree theorem}
is polynomial, of course.

Finally, even when the bijection is nice and efficient,
it might still have no interesting properties, so the
only application is the proof of the theorem.  One example
is an iterative bijection for the Rogers--Ramanujan
identity~\eqref{eq:RR-id} which is implied by the
proof in~\cite{BouP}.  It is unclear if it respects any
natural statistics which would imply a stronger result.
Thus, it is presented in~\cite{BouP} in the form of a
combinatorial proof to make the underlying algebra clear.

\smallskip

\subsection{Probabilistic/asymptotic approach}
Suppose both sets of combinatorial objects $\cP_n$ and $\cQ_n$
have well-defined \emph{limit shapes} $\pi$ and $\om$,
as $n\to\infty$.
Such limit shapes exist for various families of trees~\cite{Drm},
graphs~\cite{Lov-graph-limits}, partitions~\cite{DP2},
permutations~\cite{H+}, solid partitions~\cite{Oko},
Young tableaux~\cite{Rom}, etc.\footnote{Here the notion of
a ``limit shape'' is used very loosely, as it means very different
things in each case. }
\ts For a sequence $\{\vp_n\}$ of bijections \ts $\vp_n: \cP_n \to \cQ_n$,
one can ask about the \emph{limit bijection} \ts $\Phi:\bcP\to \bcQ$,
defined as \ts
$\lim_{n\to \infty}\vp_n$.  We can then require that~$\Phi$
satisfies certain additional structural properties.
This is the approach taken in~\cite{Pak-nature} to prove the
following result:

\begin{thm}
The Rogers--Ramanujan identity~\eqref{eq:RR-id} has no
\emph{geometric bijection}.
\end{thm}

Here the \emph{geometric bijections} are defined as compositions
of certain piecewise $\GL(2,\zz)$ maps acting on Ferrers diagrams,
which are viewed as subsets of $\zz^2$.  We first prove that the limits of
such bijections are \emph{asymptotically stable}, i.e. act
piecewise linearly on the limit shapes.  The rest of the proof
follows from existing results on the limit shapes~$\pi$ and~$\om$ on
both sides of~\eqref{eq:RR-id}, which forbid a piecewise linear map
$\Phi: \pi\to \om$, see~\cite{DP2}.

\smallskip

The next story is incomplete, yet the outlines are becoming clear.
Let \ts $\ASM(n)$ \ts be the number of \emph{alternating sign matrices}
of order~$n$, defined as the number of $n\times n$ matrices where
every row has entries in $\{0,\pm 1\}$, with row and column
sums equal to~$1$, and all signs alternate in each row and column.
Let \ts $\SLT(n)$ \ts be the number of the \emph{fully symmetric
lozenge tilings}, defined as lozenge tilings of the regular
$2n$-hexagon with the full group of symmetries~$D_6$.  Such tilings
are in easy bijection with solid partitions that fit into a \ts
$[2n\times 2n \times 2n]$ \ts box, have full group of symmetries~$S_3$,
and are self-complementary within the box (cf.~$\S$\ref{ss:bij-solid}).
Finally, let $\TSPP(n)$ be the
number of \emph{triangular shifted plane partitions} defined as
plane partitions \ts $(b_{ij})_{1\le i \le j}$ \ts of shifted shape
$(n - 1, n - 2, \ldots, 1)$, and entries  $n - i \le b_{ij} \le n$ for
$1 \le i \le j \le n - 1$.

The following identity is justly celebrated:
\begin{equation}\label{eq:ASM}
\ASM(n) \, = \, \SLT(n)\, = \,\TSPP(n)\, = \,
\frac{1! \. 4! \. 7! \. \cdots \. (3n-2)!}{n! \. (n+1)! \. \cdots \. (2n-1)!}
\end{equation}
See~\cite{Bre} for the history of the problem and
\cite[\href{https://oeis.org/A005130}{A005130}]{OEIS} for
further references.

Now, the second equality is known to have a bijective proof~\cite{MRR}.
The third equality is a major open problem (see below).

\begin{claim}
The equality \ts $\ASM(n) = \SLT(n)$ \ts has no
\emph{geometric bijection}.
\end{claim}

We now know (conjecturally) what the \emph{frozen regions} in each
case are: the circle for $\SLT$s and a rather involved sextic equation
for~$\ASM$s. The latter is an ingenuous conjecture in~\cite{CP}
(see also~\cite{CSp}), while the
former is a natural conjecture about the \emph{Arctic Circle} which remains
when the symmetries are introduced (cf.~\cite{Panova}).\footnote{While the
frozen region hasn't been established for $\SLT$s, it is known that if
exists it must be a circle (Greta Panova, personal communication).}
We are not sure in this case what do we mean by a ``geometric bijection''.
But any natural definition should imply that the two shapes are incompatible.  It would
be interesting to formalize this even before both frozen regions
are fully established.

\smallskip

There is another aspect of this asymptotic approach, which allows
us to distinguish between different equinumerous collections of
combinatorial objects with respect to some (transitive) notions of a
``good'' (canonical) bijection, and thus divide them into equivalence
classes.  This method would allow us to understand the nature of these
families and ignore superficial differences within the same class.

The prototypical example of this is a collection of over 200 objects enumerating
Catalan numbers~\cite{Sta-cat}, but there are other large such collections:
for Motzkin numbers, Schr\"oder numbers, Euler numbers~\eqref{eq:formula-Euler},
etc.  A natural approach would be to use the symmetry properties or the topology,
but such examples are rare (see, however, \cite{AST}
and~\cite{West} for two ``canonical'' bijections between Catalan objects).

In~\cite{MP}, we studied the limit averages of permutation matrices
corresponding to $\Av_n(\cf)$.  We showed that the limit surfaces
corresponding to $\Av_n(123)$ and  $\Av_n(213)$ are quite different,
even though their sizes are Catalan numbers (see also~\cite{HRS,Madras-P}).
This partly explains a well known phenomenon: there are at least \emph{nine}(!)
different explicit bijections between these two families, see~\cite{Kit},
each with its own special properties.  Evidently, there is simply no ``canonical''
bijection in this case.  See also~\cite{Ald,Dokos-P} for the asymptotic analysis
of two other interesting Catalan families.

\smallskip

\subsection{Open problems on bijections}\label{ss:bij-op}
In theory, having a direct bijection should be an exception, not
a rule, since in most cases the algebraic % and geometric
tools are simply more powerful.
In practice, combinatorialists tend to be fascinated with basic
structures reflecting certain most fundamental symmetries, where the
bijections are abound.  There are, however, a few notable examples
where the bijections have been sought for years, sometimes
for decades, with little hope of success.  Below is a very short
list from the many remarkable identities.

\smallskip

\nin
$(1)$ \. {\it Dyson's rank problem}. \ts Prove bijectively:
$p_0(5k-1)=p_1(5k-1)= p_2(5k-1)$.
Here $p_i(n)$ is the number of partitions $\la \vdash n$ such that
$\la_1-\la_1'=i$~mod~$5$, see~\cite{Dyson}.
Note that asymptotic methods are inapplicable here (for the purposes
of proving non-existence of such bijection), but there is
an elegant algebraic proof~\cite{GKS}.

\smallskip

\nin
$(2)$ \. {\it Triangular shifted plane partition}. \ts
Prove bijectively the product formula~\eqref{eq:ASM} for~$\TSPP(n)$.
We refer to~\cite{Bre,Kratt-PP} for more on the history and the context.

\smallskip

\nin
$(3)$ \. {\it Symmetry of $q,t$-Catalan numbers.} \ts Prove bijectively:
$$
F_n(x,y) \. = \. F_n(y,x), \quad \text{where} \ \
F_n(x,y) \. = \. \sum_{\pi \in \Dyck(n)} \. x^{\area(\pi)}\ts y^{\dinv(\pi)}
$$
where $\Dyck(n)$ is the set of \emph{Dyck paths} of length~$2n$ and two statistics
are defined in~\cite{Hag}.
\smallskip

\nin
$(4)$ \. {\it Linear extensions of Selberg posets}:\footnote{We introduced
Selberg posets in a 2003 NSF proposal, see solution to~\cite[Exc.~3.55]{Stanley-EC2}.}
Let $\rP(a,b,c)$ be a poset defined as in Figure~\ref{f:selberg}.  Let $e(\rQ)$ denote
the number of \emph{linear extensions} of the poset~$\rQ$. Prove bijectively:
$$e\bigl(\rP(a,b,c)\bigr) \, = \,
\frac{(a+c)!\, (b+c)! \, (2a+2b+2c+1)!}
{a!\,b!\,c!\,(a+b+c)! \,(a+b+2c+1)!}\,.
$$
See Ex.~1.11 and~3.55 in~\cite{Stanley-EC2} for generalization
and connection to the \emph{Selberg integral}.

{\small
\begin{figure}[hbt]
 \begin{center}
   \includegraphics[height=3.3cm]{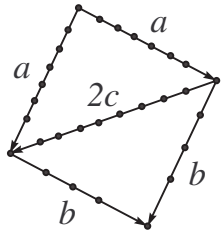}
   \caption{Selberg poset~$P(6,3,4)$.}
   \label{f:selberg}
 \end{center}
\end{figure}
}

\smallskip

\nin
$(5)$ \. {\it Standard Young tableaux of skew shape.}  Prove bijectively:
$$
\#\SYT\bigl((3a)^{2a}(2a)^a/a^a\bigr) \, = \, \frac{(7a^2)! \,\. \Phi(a)^5 \,\Phi(5a)}{\Phi(2a)^2\, \Phi(6a)}\,,
$$
where \ts $\Phi(n) = 1! \cdot 2! \ts \cdots \ts (n-1)!$ \ts is the \emph{superfactorial}
\cite[\href{https://oeis.org/A000178}{A000178}]{OEIS}.   This is a special case of two large
series of shapes recently discovered in~\cite{KO} and~\cite{MPP}.  One can also view $\SYT(\la)$
as linear extensions of the corresponding poset, and in this case it also can be derived
from the Selberg integral~\cite{KO}.

\smallskip

\begin{rem}{\rm
Let us note that posets in $(4)$ and $(5)$ have \emph{dimension two},
meaning they are defined by a set of points \ts $\bigl\{(x_1,y_1), \ldots,(x_n,y_n)\bigr\}\ssu \rr^2$,
with \ts $(x_i,y_i) \preccurlyeq (x_j,y_j)$ \ts when $x_i\le x_j$ and $y_i\le y_j$.
The formulas above allow computation of the number of linear extensions in
polynomial time.  Perhaps surprisingly, many other sequences in this paper also
count linear extensions of $2$-dimensional posets: binomial coefficients~$\binom{n}{k}$,
Catalan numbers~$C_n$, Fibonacci numbers~$F_n$, Euler numbers~$E_n$, etc.\
(see e.g.~\cite{MPP-monthly,MPP}).  

In fact, Feit's determinant formula~\cite{Feit} for $\SYT(\la/\mu)$ 
implies that the number of linear extensions can be computed in poly$(n)$ time 
for all skew shapes, where $n=|\la/\mu|$.  This raises a question if $e(\rP)$ can
be computed in polynomial time for other ``natural'' families of shapes, e.g.\ 
\emph{cross shapes} \ts $\rCr(a,b)$ defined as $[b\times b]$ minus four 
corner $[a\times a]$ squares, $b >2a$ (cf.~\cite{AdR}).  Finally, 
let us mention that computing $e(\rP)$ is $\SP$-complete for general 
$2$-dimensional posets~\cite{Dit-P}.
}
\end{rem}

\medskip

% \newpage

\section{Combinatorial interpretations}\label{s:inter}

\subsection{Complexity setting}  \label{ss:inter-complexty}
Let $\cA= \cup_n \cA_n$ be a family of \emph{combinatorial objects},
which means that membership in~$\ca_n$ can be decided in
\poly$(n)$ time (see~$\S$\ref{ss:formula-comp}).  Let
\ts $f: \cA\to \nn$ \ts be a computable function,
which we assume to be at most exponential: \ts
$f(X) \le e^{C\ts n^a}$ \ts for all $X\in \cA_n$,
and some $C,a>0$.  Let
$$\cP \. = \bigcup_{X\in \cA} \ts \cP_X
$$
be a family  of combinatorial objects parameterized by~$\cA$,
such that \ts $|\cP_X|=f(X)$.  We then say that $\cP$ is a
\emph{combinatorial interpretation} of~$f$.

In the language of computational complexity, if \ts $f$ \ts has
a combinatorial interpretation, then the problem of
computing $f(X)$ is in $\SP$.  Similarly, suppose $f_1$ and
$f_2$ have combinatorial interpretations, and that \ts
$f = f_1 - f_2$.  Then the problem of computing
$f(X)$ is in the complexity class~$\GapP$, defined exactly
to be the class of differences of two~$\SP$ functions.
It is typical in combinatorics to consider a nonnegative $\GapP$~function
and ask if it is in~$\SP$ (see examples below).

% There are currently
% no formal obstructions to this, so it is conceivable that
% this can be done in all cases.  However, as with \,
% $\PP$~vs.~$\BPP$, in many cases of interest we just
%don't know how to do that.

\smallskip

\subsection{Kronecker coefficients}  \label{ss:inter-Kron}
One especially notable case of a combinatorial interpretation
is the problem of computing the
\emph{Kronecker coefficients} of the symmetric group, which are defined by
$$\chi^\la \cdot \chi^\mu \, = \, \sum_{\nu\vdash n} \,
g(\la,\mu,\nu) \. \chi^\nu, \quad \text{where} \ \ \la,\mu \vdash n\ts,
$$
and $\chi^\la, \ts\chi^\mu, \ts\chi^\nu$ \ts are irreducible characters
of~$S_n$.

\begin{op}\label{op:kron}
Find a combinatorial interpretation for the
Kronecker coefficients
$$\bigl\{\ts g(\la,\mu,\nu), \, \la,\mu, \nu\vdash n\ts \bigr\}.$$
\end{op}

This problem was introduced by F.~D.~Murnaghan in~1938 and
it has been studied extensively in recent years, both as a problem in
algebraic combinatorics and in connection to the geometric
complexity theory.  We refer to~\cite{Bla,Pak-Panova-complexity} for
details and further references.

Now, it was shown by B\"urgisser and Ikenmeyer~\cite{BI} that
computing \ts $g(\la,\mu,\nu)$ \ts is in $\GapP$.  An elementary
proof of this is given in~\cite{Pak-Panova-complexity}:
% \begin{equation}\label{gapp_eq}
$$
g(\la,\mu,\nu) \, =  \. \sum_{\si,\om,\pi \in S_\ell} \. \sgn(\si\om\pi) \cdot
\rCT(\la-\si+\bone,\mu-\om+\bone,\nu-\pi+\bone)\.,
$$
%\end{equation}
where $\bone=(1,\ldots,1)$, the number of rows $\ell(\la),\ell(\mu),\ell(\nu)\le \ell$,
and \ts $\rCT(\al,\be,\ga)$ \ts is the number of $3$-dimensional \emph{contingency tables}
with $2$-dimensional sums $\al,\be,\ga$.  The problem is known to be $\SP$-hard~\cite{BI}.
In a recent remarkable development, it was shown in~\cite{IMW} that the \emph{positivity}
decision problem \ts $g(\la,\mu,\nu)\.{}>?\,0$ \ts is $\NP$-hard, so a combinatorial
interpretation would imply that this problem is $\NP$-complete.

\medskip

\subsection{Hamiltonian cycles in cubic graphs}  \label{ss:inter-ham}
The Kronecker coefficients problem discussed above is fundamentally an
issue of \emph{constructive subtraction} encapsulated by the $\GapP$
complexity class.  There is a similar \emph{constructive division} issue,
which is rare in complexity theory, but routine in combinatorics, where
many proofs are based on \emph{double counting}.  For example, both Joyal's
and Pitman's proofs of the Cayley formula~\eqref{eq:formula-Cayley} are by
double counting (see the first and the fourth proof in~\cite[Ch.~32]{AZ}).

\begin{thm}[Smith and Tutte~\cite{Tutte-ham}]  \label{t:Tutte-ham}
Let $e$ be an edge in a cubic graph~$G$.
Then the number $N_e(G)$ of Hamiltonian cycles in $G$ containing~$e$,
is always even.
\end{thm}

The original proof is an elegant double counting argument modulo~$2$.
This suggests the following open problem:\footnote{This problem
was suggested to us by Peter Shor  \ts
\href{https://cstheory.stackexchange.com/questions/33508}{https://tinyurl.com/y7yafneh}.}

\begin{op}
Find a combinatorial interpretation for $N_e(G)/2$.
\end{op}

For example, when $G$ is uniquely $3$-edge-colorable, we have $N_e(G)=2$
and the problem is easy.  The following constructive proof of
Theorem~\ref{t:Tutte-ham} is due to Price~\cite{Price} and Thomason~\cite{Tho}:

\begin{proof}[Proof of Theorem~\ref{t:Tutte-ham}]
Consider the set of Hamiltonian paths $P$ in $G$ with $e$ as the initial
step.  Initially, let $P$ be a Hamiltonian cycle in~$G$ minus edge $e_1$
adjacent to~$e$.  The endpoint of~$P$ that is not in~$e$ has two
adjacent edges $e_1, e_2\notin P$.  Extend $P$ by $e_2$ and
remove a uniquely determined edge which creates a new path~$P'$.
This transformation is called \emph{P\'osa rotation}.  Repeat Po\'sa
rotations until one would be forced to remove~$e$, at which point you get
a new Hamiltonian cycle. This defines a pairing of all Hamiltonian cycles.
\end{proof}

If the \emph{Price--Thomason algorithm} defined in the proof above were
polynomial, we would have a solution for the open problem.  Indeed,
for a Hamiltonian cycle in $G$ we accept it if and only if it is
lexicographically smaller than the Hamiltonian cycle it is
paired with by the algorithm.  Unfortunately, the algorithm is
exponential~\cite{Kraw} (see also~\cite{Cam}).

\smallskip

Note the strong similarities with the \emph{involution principle}
(see~$\S$\ref{ss:bij-complexity}) and the \emph{division by two}~\cite{DC}
which are all based on the same principle.  The problem and the proof
above are variations on {\sc Another Hamiltonian Cycle} problem,
which is in the complexity class~$\PPA$ conjectured to be outside of~$\PP$,
see~\cite[Ex.~10.7]{Pap} and~\cite[$\S$6.7.2]{MM}.  This suggests that
the open problem cannot be resolved in full generality (in a mathematical
sense), but perhaps some

\smallskip

We should mention that in some cases the constructive division problem
has been resolved.  Most notably, the \emph{Ramanujan congruences}
modulo~5 and~7, see~\cite[$\S$19.12]{HW}, were combinatorially
interpreted by Dyson, with a proof of his conjectural combinatorial
interpretation given later by Atkin and Swinnerton-Dyer.
See \emph{Dyson's rank problem} in~$\S$\ref{ss:bij-op} and an
alternative statistics called \emph{crank}, introduced
in~\cite{AG-crank} for congruences modulo~11, leading to many
generalizations~\cite{Mah}.

\medskip

\subsection{Combinatorial interpretation of sequences}  \label{ss:inter-seq}
Let $\{a_n\}$ be a nonnegative integer sequence that has a Wilfian formula
of type~$\WI$, i.e.\ can be computed in \poly$(n)$ time.  Then $a_n$ has
a trivial combinatorial interpretation: integers $\{1,\ldots,a_n\}$.  This
means that for sequences we need a different notion.

Let \ts $\cP=\cup\cP_n$ \ts be the set of combinatorial objects, s.t.
$|\cP_n|=a_n$ for all~$n$.  We say that $\cP$ \emph{gives a combinatorial
interpretation} for $\{a_n\}$ of type:

\smallskip
\quad $\CI$ \ts if membership in $\cP_n$ can be decided in $O(\log n)$ space.

\smallskip

\nin
For example, the $0$-$1$ ballot sequences of length $2n$ with $n$ ones
and $n$ zeroes give a combinatorial interpretation of the Catalan
number~$C_n$.  Indeed, to verify  membership it suffices to have two
counters: \ts $\#$1's and $\#0$'s, which require $O(\log n)$ space.

Note that a combinatorial interpretation of type~$\CI$ is a property
not only of the objects in~$\cP$, but also of their presentation.
For example, permutations $\si\in S_n$ in their natural representation
are not of type~$\CI$.  To make them of type~$\CI$ one can
represent them with a permutation matrix.

\medskip

\subsection{Super Catalan numbers}  \label{ss:inter-super}
The strange case of \emph{super Catalan numbers} shows both
advantage and disadvantage of the complexity approach.  They
are defined as follows:
$$
S(m,n) \. := \, \frac{(2m)! \. (2n)!}{m!\. n!\. (m+n)!}\..
$$
These were defined by E.~Catalan in 1874, who discovered that
they are integers.  This follows easily from the divisibility
properties of the factorials as well as the
\emph{von Szily identity} (1894):
$$
S(m,n) \, = \, \sum_k \. (-1)^k \ts \binom{2m}{m+k} \ts \binom{2n}{n+k}\ts.
$$
Note that $S(1,n)/2=C_n$ is the usual Catalan number.

The problem of finding a combinatorial interpretation was posed by Gessel
in~\cite{Ges-super}.  Such an interpretation is known for $m\le 3$ and
for $|m-n|\le 3$, see~\cite{CW,GX} (see also~\cite{AG,Scha-super}).
Yet, from a computational complexity point of view, Gessel also
suggested the proof idea.

\begin{thm}\label{t:super}
The numbers $S(m,n)$ have a combinatorial interpretation of type~$\CI$.
\end{thm}

\begin{proof}% [Second proof (after Gessel)]
Following Gessel~\cite{Ges-super}, we have
\begin{equation}\label{eq:super-ges}
S(m,m+\ell) \, = \, \sum_{k} \. 2^{\ell-2k} \ts \binom{\ell}{2k} \ts S(m,k)\ts.
\end{equation}
Together with the symmetry $S(m,n)=S(n,m)$ and the initial condition $S(0,0)=1$,
this recursively defines a combinatorial interpretation. Indeed, at each
recursive step a $0$--$1$ word of length $(\ell-2k)+\ell\le 2\ell$ is added.
Since $k\le \ell/2$, the number of steps to compute $S(m,n)$, where $m\le n$, is
at most $O(n\log n)$.  All such words put together, this gives a word of length
$O(n^2\log n)$. The description of each step can be verified in $O(\log \ell)$
space, giving the total space $O(\log m + \log n)$, as desired.
\end{proof}

{\small 
\begin{rem}{\rm  Finding a combinatorial interpretation
for the super Catalan numbers $S(m,n)$ is repeatedly stated
as an open problem, see e.g.~\cite[Exc.~66]{Stanley-bij}.
Gessel writes: ``it remains to be seen whether~\eqref{eq:super-ges}
can be interpreted in a `natural' way''~\cite{Ges-super}.
This was later echoed in~\cite{GX}: ``Formula~\eqref{eq:super-ges}
allows us to construct recursively a set of cardinality $S(m,n)$,
but it is difficult to give a natural description of this set.''

Given the argument above, we are somewhat puzzled as to what
exactly is an open problem, and remain deeply dissatisfied with the
``you know it when you see it'' answer.  As best as we can tell,
the unwritten goal is to represent $S(m,n)$ as the number of
certain (collections of) lattice paths, which would allow
an easy inductive proof and passing to an existing
$q$-analogue~\cite{WaZ}.  

Alternatively, one can ask about a combinatorial interpretation 
coupled with a double counting argument elucidating the product 
formula for $S(m,n)$, as in~\cite{Scha-super}.  Such notion would
be more restrictive, of course, and thus harder to obtain.  It would 
also be a more exciting discovery.  
}
\end{rem}
}
\smallskip

\subsection{Gessel sequence}  \label{ss:inter-gessel}
In~\cite{Ges-sem}, the author defined the following \emph{Gessel sequence}:
$$
b_n \, := \ 2\cdot 5^n \. - \. (3+4i)^{n} \. - \. (3-4i)^{n}, \quad \text{where} \ \ i \. = \. \sqrt{-1}\ts,
$$
see also \cite[\href{https://oeis.org/A250102}{A250102}]{OEIS}.
Note that $b_n\in \zz$ since
$$b_n \, =  \, 2\cdot 5^n \. - \. 2\ts \sum_r \. (-1)^r \ts \binom{n}{2r} \. 3^{n-2r} \ts 4^{2r}\ts,
$$
and that $b_i \ge 0$ since $|3\pm 4i|=5$.  Note also that $\{b_n\}$ is C-recursive since
$$
B(t) \. :=\, \sum_{n=0}^\infty \. b_n \ts t^n \, = \, \frac{16 \ts t \ts (1+5t)}{(1-5t)\ts (1+6t+25t^2)}
$$
This is an example of a C-recursive nonnegative sequence without an easy
combinatorial interpretation.

\begin{conj}[cf.~\cite{Ges-sem}]  \label{c:ges-seq}
Sequence \ts $\{b_n\}$ \ts has a combinatorial interpretation
of type~$\CI$.
\end{conj}

Let us briefly explain the significance of the sequence $\{b_n\}$.
The class $\CR$  of \emph{$\nn$-rational functions} is defined
to be the smallest class of GFs $F(t)=a_0 +a_1t+a_2t^2+\ldots\ts$,
such that

\smallskip

$(1)$ \ $0, t \in \ts \CR$,

$(2)$ \ $F_1, F_2 \in \ts \CR$ \ \, $\Longrightarrow$ \ \, $F_1+F_2, F_1\cdot F_2 \in \ts \CR$,

$(3)$ \ $F\in \ts \CR$, $F(0)=0$ \ \, $\Longrightarrow$ \ \, $1/(1-F) \in \ts \CR$\ts.

\smallskip

\nin
Clearly, $\CR \subseteq \zz(t)\cap\nn[[t]]$.  This is a class of GFs
for the number of words in \emph{regular languages}.  Equivalently,
this is a class of GFs for the number of accepted paths by a
\emph{finite-state automaton} (FSA).  The
\emph{Berstel--Soittola theorem} gives necessary and sufficient
conditions for a nonnegative GF to be in $\cR$~\cite{BerR}
(see also~\cite{Ges-sem}).
The corresponding sequences are exactly those that have a
combinatorial interpretation $\cP$, s.t.

\smallskip
\quad $\CII$ \ts membership in $\cP_n$ can be decided in $O(1)$ space.

\smallskip

\nin
In the case of the Gessel sequence, the Berstel--Soittola theorem implies that
\ts $B(t) \notin\CR$.  This implies that $\{b_n\}$ has no combinatorial
interpretation of type~$\CII$, i.e.\ cannot be described by a FSA,
thus making Conjecture~\ref{c:ges-seq} more challenging.

\begin{rem}{\rm
The (metamathematical) \emph{Sch\"utzenberger principle} states that
all combinatorial sequences with rational GFs must be $\nn$-rational,
see~\cite[p.~149]{BerR1}.
This all depends on the meaning of the word ``combinatorial'',
of course.  Philosophy aside, we believe the conjecture above
will be resolved positively, and plan to return to this problem
in the near future.
}\end{rem}

\begin{rem}{\rm
For general C-recursive integer sequences $\{a_n\}$, finding a combinatorial
interpretation is related to the classical \emph{Skolem problem} of deciding
if $a_n =0$ for some~$n$.  Skolem's problem is known to be \ts $\NP$-hard, but
not known to be decidable except for some special cases, see~\cite{ABV,OW}.
Since $\{a_n^2\}$ is also C-recursive, having a combinatorial
interpretation for $\{a_n^2\}$ would not be a contradiction to undecidability,
but would make it similarly curious as \textsc{Rectangular Tileability},
see~$\S$\ref{ss:tilings}.
}\end{rem}

\smallskip

\subsection{Unimodality of $q$-binomial coefficients}  \label{ss:inter-qbin}
A sequence $(a_0,a_1,\ldots,a_n)$ is called \emph{unimodal} if for some~$m$ we have
$$
a_0 \, \le \, a_1 \, \le \,  \cdots \, \le \,  a_m \, \ge \,  a_{m+1}
\, \ge \, \cdots \, \ge \, a_n\ts.
$$
Whenever a combinatorial sequence is proved unimodal, one can ask for
a combinatorial interpretation of \ts $\bigl\{a_k-a_{k-1}, 1\le k \le m\bigr\}$ \ts
and \ts $\bigl\{a_{k}-a_{k+1}, m\le k \le n-1\bigr\}$. For example, for
$a_k = \binom{n}{k}$, unimodality can be checked directly.  The differences
\ts $\binom{n}{k}-\binom{n}{k-1}$, $1\le k \le n/2$ \ts are the
\emph{ballot numbers}, generalizing Catalan numbers.

The \emph{$q$-binomial}~(Gaussian)~\emph{coefficients} are defined as:
$$
\binom{m+\ell}{m}_q
\, = \ \. \frac{(q^{m+1}-1)\. \cdots\. (q^{m+\ell}-1)}{(q-1)\.\cdots\. (q^{\ell}-1)}
\ \. = \, \, \sum_{k=0}^{\ell\ts m} \, \. p_k(\ell,m) \. q^k\ts.
$$
\emph{Sylvester's theorem} establishes unimodality of the sequence
$$
p_0(\ell,m)\ts, \, p_1(\ell,m)\ts, \, \ldots \,, \, p_{\ell\ts m}(\ell,m)\ts.
$$
This celebrated result was first conjectured by Cayley in~1856, and proved
by Sylvester using \emph{invariant theory} in a pioneering~1878
paper~\cite{Syl} (see~\cite{OHara,Pak-Panova-unim,PoRo,Pro} for modern proofs).

\begin{thm}[Pak--Panova, 2015]  \label{t:PP-qbin}
Fix $\ell,m\ge 1$.  The sequence
$$
\bigl\{ p_k(\ell,m)\ts - \ts p_{k-1}(\ell,m), \ts 1\le k \le \ell\ts m/2\bigr\}
$$
has a combinatorial interpretation of type~$\CI$.\footnote{This combinatorial
interpretation is based on O'Hara's identity~\cite{OHara} and is quite involved.
We give it on p.~9 in \ts
\href{http://www.math.ucla.edu/~pak/hidden/papers/Panova_Porto_meeting.pdf}{https://tinyurl.com/ydemhyf5}.}
\end{thm}

Note that the sequence in the theorem is a special case of Kronecker coefficients:
$$
p_k(\ell,m) \ts - \ts p_{k-1}(\ell,m) \. = \. g\bigl(m^\ell, \ts m^\ell, (m\ell-k,k)\bigr).
$$
In a roundabout way, the technical difficulties involved in the proof of
Theorem~\ref{t:PP-qbin} suggest that Open Problem~\ref{op:kron} is
unlikely to have a easy solution.  Moreover, since the combinatorial
interpretation of Kronecker coefficients in the theorem is in terms
of certain partition-labeled trees, it is unlikely that \ts
$g(\la,\mu,\nu)$ in general can be expressed as the number of certain
Young tableaux (cf.~\cite{Bla}).

\smallskip

Unimodality and related log-concavity problems are plentiful
in combinatorics, with many connections and applications to other fields;
see e.g.~\cite{Brenti} for an introduction.  While occasionally the proofs
are elegant combinatorial constructions (see e.g.~\cite{HS,Kratt-match}),
most of them are rather difficult and technical, involving fundamentally
non-combinatorial tools.  Thus, for example, it would be unreasonable
to expect a direct combinatorial proof of log-convexity of the partition
function:
$$
p(n-1) \. p(n+1) \, \ge \, p(n)^2 \quad \text{for all} \ \ n\ge 26,
$$
see~\cite{DP}.  Similarly, it would be unreasonable to expect a
combinatorial proof of the \emph{Adiprasito--Huh--Katz theorem} on
log-concavity of \ts $\bigl\{a_k(G)\bigr\}$, where $G=(V,E)$
is a simple graph and $a_k(G)$ is the number of spanning forests in~$G$
with $k$ edges~\cite{AHK}.

\bigskip

% \newpage

\section{Final Remarks}\label{s:fin-rem}

\subsection{} \label{ss:finrem-prequel}
Enumerative combinatorics is so enormous in both range and scope, that we touched
upon very few themes.  If one is to summarize our choices, we tried to explain
how to ask a good question on the subject.  This includes both the types of
questions one can ask from the complexity point of view, as well as the sources
of combinatorial sequences and combinatorial objects to study.

\subsection{} \label{ss:finrem-quote}
This paper can be viewed as a technical followup to an elegant, refreshingly
opinionated and very accessible introductory article~\cite{Zei}.  To understand
the state of art of \emph{Enumerative Combinatorics}, we refer to an excellent
monograph~\cite{Stanley-EC2}, which is remarkable in both the content
(check all the exercises!) and presentation style.

For \emph{Computational Complexity} definitions and the background,
we recommend~\cite{MM} as a fun and accessible introductory textbook.
Other good options include:  \cite{Gold,Pap} are thorough monographs,
\cite{Aar} is a beautifully written up to date introductory survey,
and~\cite{Wig} is a remarkable recent monograph-length survey of
the area with a lot of advanced mathematical content.

As of now, the complexity approach pioneered by Wilf
in~\cite{Wilf} has yet to fully blossom into a research
area of \emph{Computational Combinatorics}.\footnote{See our blog post \ts 
\url{https://igorpak.wordpress.com/2012/07/25/computational-combinatorics/}}
 However,
the fundamentals of computational complexity are clearly
as important as basic algebra and probability.  As we tried to
explain on these pages, this computational approach can change
your vision of the area and guide your understanding.

\vskip.9cm

%\newpage
{\small
\subsection*{Acknowledgements}
We are very grateful to Matthias Aschenbrenner,
Sasha Barvinok,
Art\"em Chernikov, Persi Diaconis, Sam Dittmer,
Jacob Fox, \'Eric Fusy, Bon-Soon Lin,
Pasha Pylyavskyy, Vic Reiner,
Richard Stanley and Jessica Striker for
helpful remarks and interesting conversations.

We are thankful to Michael Albert, Allen Gehret,
Fredrik Johansson, 
Alex Mennen, Marni Mishna, Cris Moore, Jay Pantone,
Martin Rubey, Gilles Schaeffer, Jeff Shallit and
Vince Vatter for reading the paper, their comments
and useful remarks.  Ira Gessel kindly suggested
the proof of Theorem~\ref{t:planar-poly}.

Special thanks to Stephen DeSalvo, Scott Garrabrant,
Alejandro Morales, Danny Nguyen, Greta Panova, Jed Yang
and Damir Yeliussizov for many collaborations and numerous
discussions, some of which undoubtedly influenced the
presentation above.
The author was partially supported by the~NSF.
}

\newpage

%\vskip.8cm

\end{document}